\documentclass[a4paper,11pt]{amsart}
\usepackage[latin1]{inputenc}

\usepackage{multicol}
\usepackage{amssymb, amsmath, amsthm}
\usepackage{latexsym}
\usepackage{dsfont}
\usepackage{graphicx}

\usepackage{curves} 
\usepackage{epic} 
\usepackage{eepic}
\usepackage{xcolor}
 
\input xy
\xyoption{all}

\usepackage{verbatim}

\usepackage{hyperref}

\newtheorem{theorem}{Theorem}[section]

\newtheorem{corollary}[theorem]{Corollary}

\newtheorem{lemma}[theorem]{Lemma}

\newtheorem{proposition}[theorem]{Proposition}

	{\par\noindent{\bf Proposition \ref{res:hiper}.}\!\!
	\nopagebreak\normalsize\it}%

	{\par\noindent{\bf Theorem \ref{result43}.}\!\!
	\nopagebreak\normalsize\it}%

	{\par\noindent{\it Idea of the proof}.  
	\nopagebreak\normalsize}%
	{\hfill\linebreak[2]\hspace*{\fill}$\circlearrowleft$}

	{\par\noindent{\it Proof of Proposition }\ref{prop:stab:smc}.  
	\nopagebreak\normalsize}%
	{\hfill\linebreak[2]\hspace*{\fill}$\circlearrowleft$}

\newenvironment{proofTHM}%
	{\par\noindent{\it Proof of Theorem }\ref{thmfn}.  
	\nopagebreak\normalsize}%
	{\hfill\linebreak[2]\hspace*{\fill}$\circlearrowleft$}

{\theoremstyle{definition}
       \newtheorem{definition}[theorem]{Definition}
	\newtheorem{claim}[theorem]{Claim}
       
       \newtheorem{remark}[theorem]{Remark}
       
       \newtheorem{parrafo}[theorem]{{\!}}  }

\numberwithin{equation}{theorem}

\newcommand{\nat}{\mathbb N}
 
\newcommand{\calo}{{\mathcal {O}}}

\DeclareMathOperator{\ord}{ord}
\DeclareMathOperator{\Ord}{Ord}

\DeclareMathOperator{\Sing}{Sing}
\DeclareMathOperator{\Spec}{Spec}

\DeclareMathOperator{\Hord}{H-ord}

\DeclareMathOperator{\In}{In}

\newcommand{\R}{{\mathcal R}}

\newcommand{\G}{{\mathcal G}}
\newcommand{\C}{{\mathcal C}}

\renewcommand{\L}{{\mathcal L}}

\newcommand{\N}{{\mathbb N}}

\newcommand{\m}{{\mathcal{M}}}
\renewcommand{\P}{\mathcal{P}}

\newcommand{\A}{\mathbb{A}}
\newcommand{\win}{\mathbf{w}\hbox{-}\mathbf{in}}

\newcommand{\id}[1]{\langle #1 \rangle}
\newcommand{\x}{{\mathbf{x}}}

\definecolor{darkpurple}{rgb}{0.28,0.24,0.55}
\definecolor{lightblue}{rgb}{0,0.75,1}

\title{On elimination of variables in the study of singularities in positive characteristic}

\author{Ang\'elica Benito}
\author{Orlando E. Villamayor U.}
 \thanks{The authors are partially supported by MTM2009-07291.}
\date{\today}

\address{Dpto. Matem\'aticas,  Universidad
Aut\'onoma de Madrid and ICMAT-UAM \\
Ciudad Universitaria de Cantoblanco, 28049 Madrid, Spain}

\email[Ang\'elica Benito]{angelica.benito@uam.es}
\email[Orlando E. Villamayor U.]{villamayor@uam.es}


\begin{document}

\maketitle

\begin{abstract} The objective of this paper is to discuss invariants of singularities of algebraic schemes over fields of positive characteristic, and to show 
how they yield the simplification of singularities.

We focus here on invariants which arise in an inductive manner, namely by successive elimination of variables. When applied to hypersurface singularities they lead us to a refinement of the notion of multiplicity. 

The main theorem proves that, under some numerical conditions 
expressed by these invariants, singularities can be simplified by blowups at centers prescribed by this refinement.

\end{abstract}

{\tableofcontents}

\section{{Introduction}}\label{intro}

\begin{parrafo}
This paper aims to the study of invariants of singularities 
over perfect fields. In the case of a hypersurface our results will provide a refinement of  the multiplicity. In this case the invariants 
were already treated in pioneering work of Cossart and Moh. However the scope of interest in this work goes beyond the hypersurface case: we present invariants for arbitrary schemes of finite type over a perfect field. The aim is to give numerical condition, by using these invariants, and to prove that if these conditions hold there is a simplification of the singularities.

 Let $X\subset V^{(d)}$ be a hypersurface embedded in a $d$-dimensional smooth scheme over a perfect field $k$. An upper semi-continuos function, say
$$mult:X\longrightarrow \mathbb{Z},$$
is defined by setting $mult(x)$ as the multiplicity of $X$ at $x$.

If $b\in \mathbb Z$ denotes the highest value achieved by the function, then the set of $b$-fold points, say $X(b)=\{x\in X\ |\ b=mult(x)\}$, is closed. We encompass this situation in a more general setting which leads us to the notion of \emph{pairs}:

\noindent {\em Pairs:} Set, as before, $V^{(d)}$ a $d$-dimensional smooth scheme over a perfect field $k$.  Fix a non-zero sheaf of ideals, say $J\subset\calo_{V^{(d)}}$, and a positive integer $b$. Here $(J,b)$ is called a \emph{pair}. It defines a closed subset in $V^{(d)}$, called the \emph{singular locus}, say
$$\Sing(J,b)=\{x\in V^{(d)}\ |\ \nu_x(J)\geq b\},$$
where $\nu_x(J)$ denotes the order of $J$ in $\calo_{V^{(d)},x}$. 
When $J=I(X)$ and $X$ is a hypersurface, then $\Sing(J,b)$ is the set of $b$-fold points of $X$.

Let $Y$ be a smooth irreducible subscheme included in $\Sing(J,b)$, and and let $V^{(d)}\overset{\pi_Y}{\longleftarrow} V^{(d)}_1$ denote  the monoidal transformation with center $Y$. Note that there is a factorization of the form
$$J\calo_{V^{(d)}_1}=I(H)^{b}J_1,$$
where $H=\pi_Y^{-1}(Y)$ denotes the exceptional hypersurface. This defines a new pair $(J_1,b)$, called the \emph{transform} of $(J,b)$. 
Similarly, an iteration of transformations, say
 \begin{equation}\label{eqin31}
\xymatrix@C=2pc@R=0pc{
(J,b) & (J_1,b) &  &(J_r,b)\\
V^{(d)} & V^{(d)}_1\ar[l]_{\pi_{Y}} &\dots  \ar[l]_{\pi_{Y_1}} &  V^{(d)}_r\ar[l]_{\pi_{Y_{r-1}}}  
}
\end{equation}
can be defined by setting $V^{(d)}_{i}\overset{\pi_{Y_i}}{\longleftarrow}V^{(d)}_{i+1}$ as the transformation with smooth centers $Y_i$ included in $\Sing(J_{i},b)$. Let $H_{i+1}$ denote the new exceptional hypersurface in $V^{(d)}_{i+1}$.

We shall always assume that such sequences are defined in such a way that the strict transforms of the $r$ exceptional hypersurfaces, say $\{H_1,\dots,H_r\}$, have normal crossings in $V^{(d)}_r$. 
A sequence of transformation, as above, is said to be a \emph{resolution} if, in addition, $\Sing(J_r,b)=\emptyset$.

In the case of hypersurfaces, a resolution of the pair $(I(X),b)$ defines an elimination of the $b$-fold points of $X$ by blowing-up centers included in the successive strict transforms of the hypersurface.

We are searching for functions defined on the class of pairs. Namely functions defined in some prescribed way along the closed sets defined by pairs.

\end{parrafo}

\begin{parrafo}{\bf The $\tau$-function on pairs}.

 An example of function on the class of pairs is that known as the $\tau$-function. For each pair $(J,b)$, over over $V^{(d)}$, a lower semi-continuous function
 $$\tau_{(J,b)}:\Sing(J,b)\longrightarrow\mathbb{Z}_{\geq0}$$
 is defined (or say, $(-1)\cdot \tau_{(J,b)}$ is upper semi-continuous). 
 The value at $x\in \Sing(J,b)$, say $\tau_{(J,b)}(x)$, is an integer that encodes fundamental information. For example:
$$\tau_{(J,b)}(x) \leq codim_x(\Sing(J,b)) \leq d,$$
where the intermediate term is the local codimension at $x$ of  $\Sing(J,b)$, and $d$ is dimension of $V^{(d)}$. So it is a lower bound of the local codimension of the closed set. Moreover:

In the case $\tau_{(J,b)}(x) = d$, the inequality says that $\Sing(J,b)=\{x\}$, and it is proved that a resolution of $(J,b)$ is achieved, locally, by blowing up the point $x$.

On the other hand, if  $\tau_{(J,b)}(x) =codim_x(\Sing(J,b))$, then, in a neighborhood of $x$, 
$\Sing(J,b)$ is smooth and a resolution of $(J,b)$ is achieved by blowing up $\Sing(J,b)$.

So from the point of view of singularities the value $\tau_{(J,b)}(x)$ is a powerful local invariant. We briefly indicate how it is defined. Let $\mathbb T_{V^{(d)},x}=\Spec(gr_M(\calo_{V^{(d)},x})$ denote the tangent space at $x\in V^{(d)}$, where $M$ denotes the maximal ideal of  $\calo_{V^{(d)},x}$. The class of $J$ in $M^b/M^{b+1}$, say
 $(J+ M^{b+1})/M^{b+1}$, spans an homogeneous ideal in $gr_M(\calo_{V^{(d)},x})$, and hence a \emph{cone}, say
$\C_{(J,b)}(x)\subset \mathbb T_{V^{(d)},x}$. 

The tangent space is viewed as a vector space. Given a cone, say $\C$, in a vector space $\mathbb V$, there is a largest subspace, say $S$, so that $S\subset \C \subset \mathbb V$, and $\C+v=\C$ for any $v\in S$ (i.e., the cone is invariant by the group of translations on $\mathbb V$
defined by vectors in $S$). This subspace is known as the subspace of vertices of the cone.

 Here we set  $\mathcal{L}_{(J,b)}(x)$ as the \emph{subspace of vertices} of $\C_{(J,b)}(x)$. Finally  $\tau_{(J,b)}(x)$ is  defined as the codimension of   $\mathcal{L}_{(J,b)}(x)$  in $\mathbb{T}_{V^{(d)},x}$. 
\end{parrafo}

\begin{parrafo}{\bf The order function, and lower dimensional H-functions.}

In practical terms, and for the sake of resolution, the invariant 
$\tau_{(J,b)}(x)$ indicates the number of variables that can be eliminated. This claim will be clarified below, but let us indicate some properties that support it. 

In the case of characteristic zero, given $x\in \Sing(J,b)$, and setting $e=\tau_{(J,b)}(x)$, then there is a smooth subscheme of dimension $d-e$, say $V^{(d-e)}\subset V^{(d)}$, together with a pair $(\overline{J},b')$, with the property 
that a resolution of $(\overline{J},b')$ defines a resolution of $(J,b)$. 

This construction can be done in a neighborhood of $x$. The advantage of this reformulation is that $\overline{J}$ is an ideal in $V^{(d-e)}$. The underlying idea is, of course, that the smaller the dimension of $V^{(d-e)}$ is, the easier it is to construct a resolution. 
Here $V^{(d-e)}$ is defined so that, locally, $\Sing(J,b)\subset V^{(d-e)}$, and moreover,
$\Sing(J,b)=\Sing(\overline{J},b')$.

We now introduce one of the main objects of interest in this work, the so called \emph{$r$-dimensional H-functions} 
$$ \Hord^{(r)}_{(J,b)} : \Sing(J,b)\longrightarrow \mathbb Q.$$
These functions are somehow subordinated to the previous $\tau$-function. In fact, if $e=\tau_{(J,b)}(x)$, then a function $\Hord^{(r)}_{(J,b)} $ is defined, at a suitable neighborhood of $x$, but only for $d-e\leq r \leq d$.

For the case $r=d$, the function $\Hord_{(J,b)} ^{(d)}$ is always defined. Set, $ \Hord_{(J,b)} ^{(d)}=\ord_{(J,b)} ^{(d)}$:
$$ \ord^{(d)}_{(J,b)} : \Sing(J,b)\longrightarrow \mathbb Q\quad \hbox{with}\quad\ord^{(d)}_{(J,b)} (y)=\frac{\nu_y(J)}{b},$$
for $y\in\Sing(J,b)$, here $\nu_y(J)$ is the order of $J$ in $\calo_{V^{(d)},y}$. Thus, the $d$-dimensional H-function is the classical order function of a pair, namely $\ord_{(J,b)}$. One checks easily that $\ord_{(J,b)} $ is upper semi-continuos. 

A property of these functions $\Hord^{(r)}_{(J,b)} $ is that they are constantly equal to 1 at points of 
$\Sing(J,b)$, in a neighborhood of $x$, except for  
\begin{equation}\label{orde}
 \Hord_{(J,b)} ^{(d-e)}:\Sing(J,b)\longrightarrow \mathbb Q.
 \end{equation}
In the case of characteristic zero, $\Hord^{(d-e)}$ admits the following description:
Set $V^{(d-e)}\subset V^{(d)}$ and $(\overline{J},b')$, as above. Recall that
$\Sing(J,b)=\Sing(\overline{J},b')$, finally set
\begin{equation}\label{ordez}
\Hord^{(d-e)}_{(J,b)} (y)=\frac{\nu_y(\overline{J})}{b'},
\end{equation} 
for $y\in \Sing(J,b)$, where $\nu_y(\overline{J})$ denotes the order of $\overline{J}$ at 
$\calo_{V^{(d-e)},y}$. This local description, which holds only in characteristic zero, indicates that (\ref{orde}) is upper semi-continuos.

However, in the case of positive characteristic, $ \Hord_{(J,b)} ^{(d-e)}(y)$  does not admit a local description as in (\ref{ordez}). Moreover, in positive characteristic (\ref{orde}) is not always upper 
semi-continuos. 

\end{parrafo}

\begin{parrafo}{\bf Objective of this work.} 

We shall indicate later how $ \Hord_{(J,b)}^{(d-e)}(y)$ is defined. The behavior of this function is rather untraceable. It has been studied in previous works (e.g. \cite{Moh}, \cite{Moh96} \cite{Co}, \cite{CP1}, \cite{HaKang}), but only as a refinement of hypersurface singularities. Our results here are directed towards singularities of arbitrary schemes over perfect fields, and the outcome is a refinement of the Hilbert-Samuel stratification.

Set $r=d-e$. As in general the functions 
 $ \Hord_{(J,b)}^{(r)} : \Sing(J,b)\longrightarrow \mathbb Q$ are not upper semi-continuos, one cannot expect a nice inductive formulation, as that in (\ref{ordez}), which is valid only in characteristic zero.
In this paper we present two functions on pairs; the \emph{roof function}, say 
\begin{equation}\label{Jroof}
 R_{(J,b)}^{(r)} : \Sing(J,b)\longrightarrow \mathbb Q
 \end{equation}
 and the \emph{floor function}, say
\begin{equation}\label{Jfloor}
 F_{(J,b)}^{(r)} : \Sing(J,b)\longrightarrow \mathbb Q.
 \end{equation}
Both $R_{(J,b)}^{(r)}$  and $F_{(J,b)}^{(r)}$ are upper semi-continuos functions, and 
$$ F_{(J,b)}^r\leq \Hord_{(J,b)}^{(r)}\leq R_{(J,b)}^{(r)}.$$
In characteristic zero we get an equality at the right hand side. In particular, $ \ord_{(J,b)}^{(r)}$ is upper semi-continuous. Moreover, the equality of all three functions  
seems to be the goal to achieve.

The objective of this paper is to explore these functions 
in positive characteristic, and to show that resolution of 
pairs $(J,b)$ is attained when one the the inequalities is an equality.

The roof functions $R^{(r)}$ will be discussed in \ref{pRoof}, and the floor functions $F^{(r)}$ in \ref{pFloor}. Some indications on the H-functions $\Hord^{(r)}_{(J,b)}$ will be discussed in \ref{pord}. A precise formulation of the previous statement is given in Theorem \ref{IntroTH}. A first step in this direction requires a harmless reformulation, in which pairs are replaced by Rees algebras (see \ref{pair2alg}). The advantage of this reformulation is that each algebra can be naturally enlarged to a new algebra which is enriched by the action of differential operators. These are called differential Rees algebras, which have been recently studied due to the strong properties they have (e.g. \cite{Hironaka05}, \cite{Kaw}, \cite{KM}, \cite{VV1}, \cite{VV4}, \cite{BV3}, \cite{BeVJapon}, \cite{Wlo2}). Moreover, for the purpose of our study we can always restrict attention to these  algebras.
\end{parrafo}

\begin{parrafo}
\noindent{\bf From Pairs to Rees algebras}.\label{pair2alg} 

A pair over $V^{(d)}$ can be viewed as an algebra.
A \emph{Rees algebra} over $V^{(d)}$ is an algebra of the form $\G=\bigoplus_{n\in\N} I_nW^n$, where $I_0=\calo_{V^{(d)}}$ and each $I_n$ is a coherent sheaf of ideals. Here $W$ denotes a dummy variable introduced to keep track of the degree, so $\G\subset\calo_{V^{(d)}}[W]$ is an inclusion of graded algebras. It is always assumed that, locally at any point of $V^{(d)}$, $\G$ is a finitely generated $\calo_{V^{(d)}}$-algebra.  Namely, that restricting to an affine set,  there are \emph{local generators}, say $\{f_{n_1},\dots,f_{n_s}\}$, so that
$${\mathcal G}=\calo_{V^{(d)}}[f_{n_1}W^{n_1}, \dots ,f_{n_s}W^{n_s} ](\subset  \calo_{V^{(d)}}[W]).$$

We now set the \emph{singular locus of} $\G=\bigoplus I_nW^n$ to be the closed set:
$$\Sing(\G):=\{x\in V^{(d)}\ | \ \nu_x(I_n)\geq n\hbox{ for each }n\in\N\}.$$

Fix a monoidal transformation $V^{(d)}\overset{\pi_C}{\longleftarrow}V^{(d)}_1$  with  center $C\subset\Sing(\G)$. For all $n\in\nat$ there is a factorization of the form
$$I_n\calo_{V_1^{(d)}}=I(H_1)^n\cdot I_n^{(1)},$$
where $H_1=\pi_C^{-1}(C)$ denotes the exceptional hypersurface. 
This defines a Rees algebra over  $ V_{1}^{(d)}$, say $\G_1=\bigoplus_{n\in\nat}I_n^{(1)}W^n$, called the \emph{transform of $\G$}. This transformation will be denoted by
\begin{equation}\label{trRA}
\xymatrix@R=0pc@C=0pc{
\G & & & & & \G_1\\
V^{(d)}  &  & & & &   V_{1}^{(d)}\ar[lllll]_{\pi_C}
}\end{equation}

A sequence of transformations will be denoted by:
\begin{equation}\label{seqintro1}
\xymatrix@R=0pc@C=0pc{
\G & & & & & \G_1 &  & & & &  &  & & & &   \G_r\\
V^{(d)}  &  & & & &   V_{1}^{(d)}\ar[lllll]_{\pi_{C}}   & & & & & \dots \ar[lllll]_{\pi_{C_1}} &  & & & &   V_{r}^{(d)}\ar[lllll]_{\pi_{C_{r-1}}}  \\
}
\end{equation}
and herein we always assume that the exceptional locus of the composite morphism $V^{(d)}  \longleftarrow V_r^{(d)}$, say $\{H_1,\dots,H_r\}$, is a union of hypersurfaces with only normal crossings in $V^{(d)}_r$.
A sequence (\ref{seqintro1}) is a \emph{resolution of} $\G$ if in addition $\Sing(\G_r)=\emptyset$.

A pair $(J,b)$ over $V^{(d)}$ defines an algebra, say 
$${\mathcal G}_{(J,b)}=\calo_{V^{(d)}}[JW^b]. $$
Note here that $\Sing(J,b)=\Sing({\mathcal G}_{(J,b)})$, and that
setting $\G= {\mathcal G}_{(J,b)} $ at (\ref{trRA}), then 
$\G_1= {\mathcal G}_{(J_1,b)} $, where $(J_1,b)$ denotes the transform of $(J,b)$. So algebras appear as a naive reformulation of pairs, and a resolution of ${\mathcal G}_{(J,b)}$ is the same as  a resolution of $(J,b)$. Furthermore, all the previous discussion, developed in terms of pairs, has a natural analog for Rees algebras.
We first discuss, briefly,  the reformulation of the $\tau$-invariant for algebras over $V^{(d)}$.

 \end{parrafo}

\begin{parrafo}{\bf Hironaka's $\tau$-invariant}\label{tauG}. 
Recall that 
${\mathbb T}_{V^{(d)},x}=\Spec(gr_M(\calo_{V^{(d)},x}))$.
An homogeneous ideal $\In_x(\G)$ in $gr_M(\calo_{V^{(d)},x})$ is defined by $\G$ at $x\in \Sing(\G)\subset V^{(d)}$. The \emph{tangent cone}, say $\mathcal C_{\G}(x)\subset \mathbb T_{V^{(d)},x},$ will be the cone defined by this ideal, and  $\L_{\G,x}\subset \mathcal C_{\G}(x)$ will denote the \emph{subspace of vertices} of $ \mathcal C_{\G}(x)$. Finally, set 
 $$\tau_{\G}:\Sing(\G)\longrightarrow\mathbb{Z}_{\geq0}$$
by setting $\tau_{\G}(x)$ as the codimension of the subspace
 $\L_{\G,x}$ in ${\mathbb T}_{V^{(d)},x}$.
 
 For the particular case of $\G=\G_{(J,b)}$, we get $\tau_{\G}(x)=\tau_{(J,b)}(x)$ at any $x\in \Sing(J,b)=\Sing(\G)$.
 \end{parrafo}

\begin{parrafo}\label{rp23}{\bf Rees algebras and differential structure}. 
There are various advantages in formulating invariants in terms of Rees algebras as opposed to their formulation in terms of pairs. One of them arises when studying Rees algebras with a natural compatibility with differential operators. 

A Rees algebra  ${\mathcal G}=\bigoplus_{n\geq 0}I_nW^n$ over  $V^{(d)}$ is said to be a {\em differential Rees algebra} if locally, say over any open affine set, 
$D_{r}(I_n)\subset I_{n-r},$
 for any index $n$ and for any differential operator $D_r$  of order $r<n$.

If this property holds for all $k$-linear differential operators, then we say that $\G$ is an \emph{absolute differential Rees algebra} over the smooth scheme $V^{(d)}$. When a smooth morphism 
 $V^{(d)}\overset{\beta}{\longrightarrow}  V^{(d')}$ is fixed, and the previous property holds 
for differential operators which are $\calo_{ V^{(d')}}$-linear, or say, 
$\beta$-relative operators, then $\G$ is said to be a \emph{$\beta$-relative differential Rees algebra}, or simply \emph{$\beta$-differential}.
\begin{proposition}
Every Rees algebra $\G$ over $V^{(d)}$ admits an extension to a new Rees algebra, say $\G\subset Diff(\G)$, so that $Diff(\G)$ is a differential Rees algebra. It has the following properties:
\begin{enumerate}
\item $Diff(\G)$ is the smallest differential Rees algebra containing $\G$.

\item $\Sing(\G)=\Sing(Diff(\G))$.

\item The equality in (2) is preserved by transformations. In particular, any resolution of $\G$ defines a resolution of $Diff(\G)$, and the converse holds.
\end{enumerate}
\end{proposition}

The property in (3) says that, for the sake of defining a resolution of $\G$, we may always assume that it is a differential Rees algebra. This is an important reduction because, as we shall see, differential Rees algebras have very powerful properties. Further details can be found in \cite[Theorems 3.2 and 4.1]{VV1}.
\end{parrafo}

\begin{parrafo}{\bf Transversal projections and elimination}\label{trans}.
Once we fix a closed point $x\in V^{(d)}$ it is very simple to construct, for any positive integer $d'\leq d$, a smooth scheme $V^{(d')}$ together with a smooth morphism $\beta:V^{(d)}\longrightarrow V^{(d')}$ (a projection), at least when restricting $V^{(d)}$ to an \'etale neighborhood of $x$. The claim follows, essentially, from the fact that $(V^{(d)},x)$ is an \'etale neighborhood of $(\mathbb{A}^{(d)},\mathbb{O})$; and plenty of smooth morphisms (in fact, plenty of surjective linear transformations) $(\mathbb{A}^{(d)},\mathbb{O})\longrightarrow (\mathbb{A}^{(d')},\mathbb{O})$ can be constructed. 
Note that if $\{x_1,\dots,x_d\}$ is a regular system of parameters at $\calo_{V^{(d)},x}$, then $(V^{(d)},x)$ is an \'etale neighborhood of $\mathbb{A}^{(d)}_k=\Spec(k[x_1,\dots,x_d])$ at the origin.

Furthermore, given a subspace $S$ of dimension $d-d'$ in  ${\mathbb T}_{V^{(d)},x}$, one can easily construct
$ V^{(d')}$ and a smooth $\beta:V^{(d)}\longrightarrow V^{(d')}$ so that 
$ker (d(\beta)_x)=S$ (here $d(\beta)_x:{\mathbb T}_{V^{(d)},x} \to {\mathbb T}_{V^{(d')},\beta(x)}$ is a surjective linear transformation).

Fix now a differential Rees algebra over $V^{(d)}$ and a closed point $x\in\Sing(\G)$. Recall here that $\tau_\G(x)=e$ is the codimension of $\mathcal{L}_{\G,x}(\subset\mathcal{C}_{\G,x})$ in the tangent space. Set $d'$ so that $d \geq d'\geq d-e$. 

For $d'$ in these conditions we say that a smooth morphism  $\beta:V^{(d)}\longrightarrow V^{(d')}$ is {\em transversal} to $\G$ at $x$ if 
$$ker(d\beta)_x\cap\mathcal{L}_{\G,x}=\mathbb{O}.$$ 

This condition is open (it holds at points in a neighborhood of $x$), and a smooth morphism  $\beta:V^{(d)}\longrightarrow V^{(d')}$ is said to be {\em transversal} to $\G$, if this 
condition holds at any point of $\Sing(\G)$.

Since $\G$ is a differential Rees algebra, it is, in particular, a $\beta$-differential Rees algebra. 

\begin{proposition}\label{elim} Assume that $\beta:V^{(d)}\longrightarrow V^{(d')}$ is transversal to $\G$, and that $\G$ is $\beta$-differential.
Then a Rees algebra $\R_{\G,\beta}$ is defined 
over $V^{(d')}$, say
\begin{equation}\label{lec1}
\xymatrix@C=2.5pc@R=-0.15pc{
\G & \R_{\G,\beta}\\
V^{(d)}\ar[r]^{\beta} & V^{(d')}
}
\end{equation}
(i.e., $\R_{\G,\beta}\subset\calo_{V^{(d')}}[W]$ ) with the following properties:
\begin{enumerate} 
\item The natural lifting of $\R_{\G,\beta}$, say $\beta^*(\R_{\G,\beta})$, is a subalgebra of $\G$. (\cite[Theorem 4.13]{VV4})

\item $\beta({\Sing(\G)}) \subset \Sing(\R_{\G,\beta})$, and moreover,
$\beta|_{\Sing(\G)}:\Sing(\G)\longrightarrow \Sing(\R_{\G,\beta})$ defines  a set theoretical bijection of $\Sing(\G)$ with its image.  (\cite[1.15 and Theorem 4.11]{VV4}, or \cite[7.1]{BV3}).

\item Given a smooth sub-scheme $Y\subset\Sing(\G)$, then $\beta(Y)(\subset\Sing(\R_{\G,\beta}))$ is isomorphic to $Y$. In particular $Y$ defines a transformation of $\G$ and also of $\R_{\G,\beta}$. (\cite[Theorem 9.1 (i)]{BV3}).

\item (\cite[Theorems 10.1 and 9.1]{BV3}). A smooth center $Y\subset\Sing(\G)$ defines a commutative diagram
\begin{equation}\label{lec2}\xymatrix@R=-0.15pc@C=3pc{
\G &  \G_1\\
V^{(d)}\ar[dddd]^{\beta} & V_1^{(d)}\ar[l]_{\pi_Y}\ar[dddd]^{\beta_1}\\
\\
\\
\\
V^{(d-1)} & V_1^{(d-1)}\ar[l]_{\pi_{\beta(Y)}}\\
\R_{\G,\beta} & (\R_{\G,\beta})_1
}\end{equation}
where $\G_1$ and $(\R_{\G,\beta})_1$ denote the transforms of 
$\G$ and $\R_{\G,\beta}$  respectively.  Here $\beta_1$ is defined in the restriction of $V_1^{(d)}$ to a neighborhood of $\Sing(\G_1)$, and this diagram has the following properties:

{\rm(4a)} $V_1^{(d)}\overset{\beta_1}{\longrightarrow} V_1^{(d')}$ is transversal to 
$\G_1$ and $\G_1$ is $\beta_1$-differential. In particular, we get:
$$\xymatrix@C=2.5pc@R=0pc{
\G_1 & \R_{\G_1,\beta_1}\\
V_1^{(d)}\ar[r]^{\beta_1} & V_1^{(d')}
}$$

{\rm (4b)}   $(\R_{\G,\beta})_1$ coincides with the algebra 
$\R_{\G_1,\beta_1}$, defined by $\beta_1$.
\end{enumerate}
\end{proposition}

The algebra $\R_{\G,\beta}$ over $ V^{(d')}$ defined by a transversal
 $\beta:V^{(d)}\longrightarrow V^{(d')}$ as in (\ref{lec1}), is called the 
 {\em elimination algebra} of $\G$ defined by $\beta$.

\end{parrafo}

\begin{parrafo}\label{cuad}
The previous Proposition says that given 
 $\beta:V^{(d)}\longrightarrow V^{(d')}$ transversal to $\G$, and if 
 $\G$ is a $\beta$-differential Rees algebra (e.g., if $\G$ is an absolute differential Rees algebra), then
an arbitrary sequence of monoidal transformations, say
\begin{equation}\label{cmut}\xymatrix@R=0pc{
\G & \G_1  & & \G_r\\
V^{(d)} & V_1^{(d)}\ar[l]_{\pi_{Y}} & \dots\ar[l]_{\pi_{Y_1}} &V_r^{(d)}\ar[l]_{\pi_{Y_{r-1}}}
}
\end{equation}

gives rise to a diagram, say
\begin{equation}\label{conmut}
\xymatrix@R=0pc@C=3.5pc{
\G & \G_1  & & \G_r\\
V^{(d)}\ar[dddd]_ {\beta} & V_1^{(d)}\ar[l]_{\pi_{Y}}\ar[dddd]_ {\beta_1} & \dots\ar[l]_{\pi_{Y_1}} &V_r^{(d)}\ar[dddd]_ {\beta_r}\ar[l]_{\pi_{Y_{r-1}}}\\
\\
\\
\\
V^{(d')} & V_1^{(d')}\ar[l]_{\pi_{\beta(Y)}} & \dots\ar[l]_{\pi_{\beta_1(Y_1)}} &V_r^{(d')}\ar[l]_{\pi_{\beta_{r-1}(Y_{r-1})}}\\
\R_{\G,\beta} & (\R_{\G,\beta})_1 & & (\R_{\G,\beta})_r
}\end{equation}
where:
\begin{enumerate}
\item For any index $i$, there is an inclusion $(\R_{\G,\beta})_i\subset \G_i$.

\item Every $\beta_i$ is transversal to $\G_i$ and $\G_i$ is $\beta_i$-differential.

\item $V_i^{(d')}\overset{\pi_{\beta(Y_i)}}{\longleftarrow} V^{(d')}_{i+1}$ denotes the transformation with center $\beta_i(Y_i)$ which is isomorphic to $Y_i$.

\item $(\R_{\G,\beta})_i=\R_{\G_i,\beta_i}$ where the later denotes the elimination algebra of $\G_i$ with respect to $\beta_i$.

\item  $\beta_i(\Sing(\G_i))\subset \Sing((\R_{\G,\beta})_i)$, and $\beta_i|_{\Sing(\G_i)}:\Sing(\G_i)\longrightarrow \beta_i(\Sing(\G_i))$ is an identification.  In the characteristic zero case, the previous inclusions are equalities, but in positive characteristic, in general, only the inclusion holds.
\end{enumerate}
\end{parrafo}

\begin{remark} It is convenient to start with a differential Rees algebra $\G$, and we know that this can be done for free. In fact, in this case,  locally at any closed point $x\in\Sing(\G)$, one can construct a smooth scheme $V^{(d')}$ and a smooth morphism $\beta:V^{(d)}\longrightarrow V^{(d')}$, and then $\G$ will always be a $\beta$-differential Rees algebra. In fact an absolute differential Rees algebra is always relative differential.
\end{remark}

\begin{parrafo}{\bf Order and lower dimensional H-functions and Rees algebras}.

Fix a Rees algebra $\G=\bigoplus I_nW^n$ over a $d$-dimensional smooth scheme $V^{(d)}$. 
Lower dimensional H-functions 
$$\Hord^{(r)}(\G):\Sing(\G)\longrightarrow \mathbb{Q}_{\geq 0}$$ 
are defined for $r$ in a certain range $d'\leq r\leq d$ (for some $d'\leq d$). For $r=d$, the function $\Hord^{(d)}$ is usually denoted simply by $\ord$, we will use this notation along this work. In this case,
\begin{equation}\label{ordG}
\ord(\G)(x)=\min\Big\{\frac{\nu_x(I_n)}{n}\ |\ n\in\mathbb{N}\Big\},
\end{equation}
with $x\in\Sing(\G)$. 

\vspace{0.2cm}

\noindent {\bf Facts}:
\begin{itemize}
\item If $\tau_{\G}(x)\geq 1$, then the function is constantly equal to 1 in an open neighborhood of $x$.
\item If $\tau_{\G}(x)\geq e$, then, in a neighborhood of $x$, the functions are defined in a range $d-e\leq r\leq d$. Moreover,  $\Hord^{(r)}(\G)$ is equal to $1$ in a neighborhood of $x$ for $r\geq d-e+1$.
\item A particular feature is that, in general, the function $\ord(\G)$ is upper semi-continuous.
\item In contrast, if $\tau_{\G}(x)\geq e$ the function $\Hord^{(d-e)}(\G)$ might not be upper semi-continuous.
\end{itemize}

\end{parrafo}

\begin{parrafo}{\bf On the upper-bound function}\label{pRoof}. 

Fix a differential Rees algebra $\G$. We make use of property (1) in Proposition \ref{elim} to define, for a given sequence of transformations (\ref{cmut}), new functions:
$$\Ord^{(d-e)}(\G_i):\Sing(\G_i)\longrightarrow \mathbb{Q}_{\geq 0}.$$
To this end fix $\beta:V^{(d)}\longrightarrow V^{(d')}$ transversal to $\G$. This defines a sequence (\ref{conmut}). Set:
$$\Ord^{(d-e)}(\G_i)(x)=\ord((\R_{\G,\beta})_i)(\beta_i(x)).$$

\begin{claim} For any $x\in\Sing(\G)$:
$$\Hord^{(d-e)}(\G_i)(x)\leq \Ord^{(d-e)}(\G_i)(x).$$
That is, the upper semi-continuous function $\Ord^{(d-e)}$ provides an upper bound for our H-function. 
\end{claim} 

\end{parrafo}

\begin{parrafo}The functions $\Ord^{(d-e)}(\G_i)$ have been studied in \cite{BV3}.
The main results are:

(1) The functions are intrinsic to the sequence of transformations 
(\ref{cmut}). Namely, they are independent of the choice of 
the transversal morphism $\beta : V^{(d)}\to V^{(d')}$ in (\ref{conmut}). (\cite[Theorem 10.1]{BV3})

(2) In the characteristic zero case, both functions $\Hord^{(d-e)}(\G)$ and $\Ord^{(d-e)}(\G)$ coincide. In particular, in such case the function $\Hord^{(d-e)}(\G)$ is upper semi-continuous.

(3) It is proved in \cite[10.4]{BV3} that a sequence (\ref{conmut}) can be constructed so that the elimination algebra $(\R_{\G,\beta})_r$ is \emph{monomial}, i.e., it is defined by an invertible sheaf of ideals supported on the exceptional locus, say
\begin{equation}\label{mon}
(\R_{\G,\beta})_r=\calo_{V_r^{(d-e)}}[I(H_1)^{\alpha_1}\dots I(H_r)^{\alpha_r}W^s]
\end{equation}
where $H_i$ is the strict transform of the exceptional component introduced by $\pi_{\beta_{i-1}(Y_{i-1})}$.

\begin{remark} 
In characteristic zero, it is easy to extend a sequence which is in the monomial case to a resolution of singularities. This extension can be constructed by choosing centers in a simple combinatorial manner. However this is not the case in positive characteristic.

\end{remark}

\end{parrafo}
\begin{parrafo}{\bf On the lower-bound function and tamed H-functions}.\label{pFloor}

A remarkable property of the H-functions, which was studied in \cite{BeV1}, is that is that they enable us to assign, to a sequence of transformations of $\G$, a monomial algebra which turns out to be a useful tool in the study of singularities. The monomial algebra assigned to (\ref{conmut}) is
\begin{equation}\label{leqmn}
\m_r W^s=\calo_{V_r^{(d-e)}}[I(H_1)^{h_1}\dots I(H_r)^{h_r}W^s],
\end{equation}
where each exponent $h_i$ is defined by setting, for $i=0,\dots,r-1$:
$$\frac{h_{i+1}}{s}=\Hord^{(d-e)}(\G_i)(\xi_{Y_i})-1,$$
where $\xi_{Y_i}$ denotes the generic point of $Y_i$.

\begin{claim}
The following inequalities hold for any sequence as (\ref{cmut}):
$$\ord(\m_rW^s)(x)\leq \Hord^{(d-e)}(\G_r)(x)\leq \Ord^{(d-e)}(\G_r)(x)$$
and all $x\in\Sing(\G_r)$. 
\end{claim}

So at least our untraceable H-function fits between two upper semi-continuous functions. Next Theorem gives conditions under which the H-function has a nice tame behavior, which will lead us to resolution.

\begin{theorem}\label{IntroTH}
Consider a sequence of transformations (\ref{conmut}) with the property in (\ref{mon}), namely that the elimination algebra is monomial. 
Set $\m_rW^s$ as in (\ref{leqmn}). If for any $x\in\Sing(\G_r)$
\begin{equation}\label{eqTHs}
 \Hord^{(d-e)}(\G_r)(x)=\ord(\m_rW^s)(x)\quad\hbox{ or }\quad \Hord^{(d-e)}(\G_r)(x)= \Ord^{(d-e)}(\G_r)(x),
 \end{equation}
then the combinatorial resolution of $\m_rW^s$ can be lifted to a resolution of $\G_r$.
\end{theorem}

\begin{definition}\label{IntroDef}
 $\G_r$ is said to be in the \emph{strong monomial case} (or say $\Hord^{(d-e)}(\G_r)$ is \emph{tamed}) when the setting of the previous theorem holds, namely one of the equalities in (\ref{eqTHs}) holds at any $x\in\Sing(\G_r)$. 
\end{definition}

This Theorem extends \cite[Theorem 8.13]{BeV1}, where the case $e=1$ is treated. This theorem and its proof will be address here in Theorem \ref{thmsmc} (see also Remark \ref{rmksmc}).
\end{parrafo}

\begin{parrafo}{\bf On the definition of the functions $\Hord^{(d-e)}$}\label{pord}.

We have indicated, so far, that these H-functions can be bounded between two upper semi-continuos functions. 
The results in this paper aim to the explicit  computation of the function in positive characteristic. This will be achieved through the notion of \emph{presentations}. These, in turn, will give us explicit formulations which will enable us to obtain our results, such as 
Theorem \ref{IntroTH}.

To clarify our strategy, let us first indicate what is known in the case $d-e=d-1$. Given (\ref{conmut})), and locally at any point $x\in\Sing(\G_r)$, a presentation, say
\begin{equation}\label{inpre}
\P=\P(\beta_r,z,f_n(z))
\end{equation}
is defined, where:
\begin{itemize}
\item $\{z=0\}$ is a $\beta_r$-section (i.e.,$\{z=0\}$ is a section of $\beta_r:V_r^{(d)}\longrightarrow V_r^{(d')}$).
\item  $f_n(z)=z^n+a_1z^{n-1}+\dots+a_n\in\calo_{V^{(d-1)}}[z]$ and $f_n(z)W^n\in\G_r$.
\end{itemize}
Moreover, a presentation can be defined at $x$ so that:
\begin{equation}\label{inVOR}
\Hord^{(d-1)}(\G_r)(x)=\min_{1\leq j\leq n}\Big\{\frac{\nu_{\beta_r(x)}(a_j)}{j},\ord((\R_{\G,\beta})_r)(\beta_r(x))\Big\}.
\end{equation}

Presentations of the form (\ref{inpre}) were the tool which lead us to the proof of Theorem \ref{thmqG} for $e=1$ (see  \cite[Theorem A.7]{BeV1}). 
We shall indicate in \ref{tau2} that a straightforward consequence of the previous result is the following: 

Suppose that $d'=d-e$ now for $e>1$. Locally at any $x\in\Sing(\G_r)$, there is a presentation, say
$$\P=\P(\beta_r,z_1,\dots,z_e,f_{n_1}(z_1),\dots,f_{n_e}(z_e))$$
where:
\begin{itemize}
\item $\{z_1=\dots=z_e=0\}$ is a $\beta_r$-section, 
\item  $f_{n_i}(z_i)W^{n_i}\in\G_r$, and for each index, they have of the form:
\begin{equation}\label{pol}
\begin{array}{l}
f_{n_1}(z_1)=z_1^{n_1}+a_1^{(1)}z_1^{n_1-1}+\dots+a_{n_1}^{(1)}\in\calo_{V^{(d-1)}}[z_1],\\
f_{n_2}(z_2)=z_2^{n_2}+a_1^{(2)}z_2^{n_2-2}+\dots+a_{n_2}^{(2)}\in\calo_{V^{(d-2)}}[z_2],\\
\ \ \ \vdots\\
f_{n_e}(z_r)=z_r^{n_e}+a_1^{(e)}z_r^{n_e-1}+\dots+a_{n_e}^{(e)}\in\calo_{V^{(d-r)}}][z_e],
\end{array}
\end{equation}
\end{itemize}

However, this form of presentation falls short to provide an expression similar to that in (\ref{inVOR}). This weakness is due to the fact that the coefficients are not in dimension $d-e$.

In this paper we overcome this difficulty. We prove the existence of  \emph{simplified presentations}, say
$$s\P=s\P(\beta_r,z_1,\dots,z_e,f_{n_1}(z_1),\dots,f_{n_e}(z_e)),$$
with the additional condition:
\begin{equation}\label{pol2}
\begin{array}{l}
f_{n_1}(z_1)=z_1^{n_1}+a_1^{(1)}z_1^{n_1-1}+\dots+a_{n_1}^{(1)}\in\calo_{V^{(d-e)}}[z_1],\\
f_{n_2}(z_2)=z_2^{n_2}+a_1^{(2)}z_2^{n_2-2}+\dots+a_{n_2}^{(2)}\in\calo_{V^{(d-e)}}[z_2],\\
\ \ \ \vdots\\
f_{n_e}(z_r)=z_r^{n_e}+a_1^{(e)}z_r^{n_e-1}+\dots+a_{n_e}^{(e)}\in\calo_{V^{(d-e)}}][z_e].
\end{array}
\end{equation}
Note that, as oppose to (\ref{pol}), all coefficients are in dimension $d-e$, i.e., $a_{j_i}^{(i)}\in\calo_{V^{(d-e)}}$. This enables us to extend the previous assertion. Namely that
the value of the H-function at $x$ is:
$$\Hord^{(d-e)}(\G_r)(x)=\!\!\!\!\min_{
\begin{array}{c}
\scriptstyle 1\leq j_i\leq n_i,\\
\scriptstyle 1\leq i\leq e
\end{array}}\!\!\!\!\Big\{\frac{\nu_{\beta_r(x)}(a_{j_i}^{(i)})}{j_i},\ord((\R_{\G,\beta})_r)(\beta_r(x))\Big\}.$$

The construction of this kind of presentations appears in Theorem \ref{sepvar}. Simplified presentations will allow us to view each equation in (\ref{pol2}) as a polynomial in one distinct variable. This enables us to extend arguments known for $d'=d-1$ to arbitrary $d'$.

In particular, simplified presentations  enable us to extend the main results in \cite{BeV1}, and give numerical conditions in terms of the H-functions, under which the sequence (\ref{conmut}) can be easily extended to a resolution.

The main tool used throughout this paper are the differential Rees algebras and their properties. The hope is that a better understanding of these algebras will lead to the conditions in Theorem \ref{IntroTH}, and hence to resolution of singularities. Here we mean resolution {\em a la Hironaka}, namely by successive blow-ups along centers included in the closed Hilbert-Samuel stratum (i.e., not in the simplified form 
introduced in \cite{EncVil99}).

 As an easy consequence, this will provide a simple proof of embedded resolution,  {\em a la Hironaka}, of $2$-dimensional scheme following the arguments in \cite{BeV1}. This last result has also been obtained by others authors in \cite{Ab1}, \cite{CJS} (arithmetic case!), \cite{Cut3}, \cite{HaWa}.
\end{parrafo}

\section{Weak equivalence}\label{sec3}

\begin{parrafo}
The word invariant was used in the previous discussion. In this section we specify the meaning of {\em invariant} at a point. The clarification of this concept is important if we expect to apply them in a useful manner. Particularly to distinguish the discussion 
in this paper, which aims at a refinement of the Hilbert Samuel stratification, from the previous works directed to a refinement of the notion of multiplicity.

Here an invariant of a singular point will be a value assigned to it, and subject to very precise and rigid conditions. It will be proved that the value of the H-function at a singular fulfills this conditions, and hence define an invariant of the singular point.

Briefly speaking, the conditions impose some form of compatibility with monoidal transformations. The notion of invariant grows from an equivalence relation defined among the class of Rees algebras over $V^{(d)}$.
\end{parrafo}

\begin{parrafo} Fix on the smooth scheme $V^{(d)}$  a set of smooth hypersurfaces with normal crossings, say $E=\{H_1,\dots,H_r\}$. Let $\G=\bigoplus I_nW^n$ be a Rees algebra over $V^{(d)}$. Let now
\begin{equation}\label{eqPB}
V^{(d)}\overset{\pi}{\longleftarrow}U
\end{equation}
be a smooth morphism.

There is a natural notion of pull-backs of the Rees algebra $\G$ and the set $E$, say to a Rees algebra $\G_U$ and a set $E_U$. Here $E_U$ is the set of the pull-backs of the hypersurfaces in $E$. The Rees algebra $\G_U$ is defined by: $\G_U=\bigoplus (I_n)_UW^n$, the natural lift of $\G$ to $U$. This will be noted by 
$$\xymatrix@R=.1cm{
V^{(d)} & U\ar[l]_{\ \ \ \pi}\\
\G, E & \G_U,E_U
}$$
Observe here that the singular locus of the Rees algebra $\G$ is compatible with pull-backs, i.e.,
$$\Sing(\G_U)=\pi^{-1}(\Sing(\G)).$$

\begin{definition}
Given $\G$ and $E$ as above, a \emph{local sequence} of $\G$ and $E$ is a sequence
$$\xymatrix@R=.1cm{
V^{(d)} & \widetilde{V}^{(d)}_1\ar[l]_{\ \ \pi_1} &\dots \ar[l]_{\ \ \pi_2} & \widetilde{V}^{(d)}_r\ar[l]_{\ \pi_r}\\
\G, E & \G_1,E_1 & & \G_r, E_r
}$$
where each $\xymatrix{\widetilde{V}_i^{(d)} & \widetilde{V}^{(d)}_{i+1}\ar[l]_{\ \ \ \pi_{i+1}}}$ is either a pull-back or a monoidal transformation along a center $C_i\subset\Sing(\G_i)$ which has  normal crossing with the union of the exceptional hypersurfaces in $E_i$, for $i=0,\dots,r-1$. In this latter case we set $E_{i+1}$ to be the union of the strict transforms of hypersurfaces in   $E_{i}$, together with the the exceptional locus of the monoidal transformation. Here we denote $\widetilde{V}_0^{(d)}=V^{(d)}$.
\end{definition}
\end{parrafo}

\begin{parrafo}

Fix a smooth scheme $V^{(d)}$ and a Rees algebra $\G=\bigoplus I_nW^n$. Note that the closed set attached to $\G$, say $\Sing(\G)$, coincides with that attached to the Rees algebra $\G^{(2)}=\bigoplus I_{2n}W^{2n}$, i.e., $\Sing(\G)=\Sing(\G^{(2)})$. 
Moreover, the same holds after an arbitrary local sequence of ($V^{(d)}$, $\G$,$\varnothing$) and $(V^{(d)},\G^{(2)},\varnothing)$.This example leads to a more general definition, the notion of weak equivalence.

\begin{definition}
Fix two Rees algebras $\G$ and $\G'$ and a set of exceptional hypersurfaces $E$  in the smooth scheme $V^{(d)}$. We say that $\G$ and $\G'$ are \emph{weakly equivalent} if:
\begin{enumerate}
\item[(i)] $\Sing(\G)=\Sing(\G')$.
\item[(ii)] A local sequence of $\G$, say:
$$\xymatrix@R=.1cm{
V^{(d)} & \widetilde{V}^{(d)}_1\ar[l]_{\ \ \pi_1} &\dots \ar[l]_{\ \ \pi_2} & \widetilde{V}^{(d)}_r\ar[l]_{\ \pi_r}\\
\G, E & \G_1,E_1 & & \G_r, E_r
}$$
defines a local sequence of $\G'$ (and vice versa), and $\Sing(\G_i)=\Sing(\G_i')$ for $i=0,\dots,r$.
\end{enumerate}
\end{definition}

\begin{remark}
Note that if $\G$ and $\G'$ are weakly equivalent as before, then also their transforms $\G_r$ and $\G_r'$ are weakly equivalent. So the weak equivalence is preserved by a local sequence.
\end{remark}

\begin{theorem}\label{rmkweq}\

\begin{enumerate}
\item Fix two Rees algebras $\G$ and $\G'$ with the same integral closure. Then, $\G$ and $\G'$ are weakly equivalent.

\item Fix a Rees algebra $\G$. Let $Diff(\G)$ be the differential Rees  algebra attached to $\G$ (see \ref{rp23}). Then $\G$ and $Diff(\G)$ are weakly equivalent.
\end{enumerate}
\end{theorem}

\begin{proof}
\begin{enumerate}
\item See \cite[Proposition 5.4]{EncVil06}.
\item Follows from Giraud's Lemma. See also \cite[Theorem 4.1]{EncVil06}.
\end{enumerate}
\end{proof}
\end{parrafo}

\begin{parrafo}
We address here the notion of invariant, as needed in our development. Given $\G$ over $V^{(d)}$, an {\em invariant} attached to $x\in \Sing(\G)$ is a value, say
$\gamma(x, \G),$
which is subject to the condition:
$$(*)\qquad\gamma(x, \G)=\gamma(x, \G')$$
whenever $\G$ and $\G'$ are weakly equivalent when restricted to some neighborhood of $x$.

\vspace{0.2cm}

A first example of invariant is that of the value $\tau_{\G}(x)$ in \ref{tauG}.  It will be proved here in Theorem \ref{theo} (1) that the value $\Hord^{(r)}(\G)(x)$ is an {\em invariant}, for any point $x\in \Sing(\G)$. 
\end{parrafo}

\section{The slope of a hypersurface and the weak equivalence}\label{sec5}
In this section, we restrict attention to the value of the $d-1$-dimensional H-function at a point in the highest multiplicity locus of a hypersurface. In such context, this value is known as the \emph{slope}. We present here the slope as a refinement of the multiplicity at the point (see Definition \ref{dfsl}).
We firstly define it in terms of a transversal projection, and then we show that it is an invariant and hence independent of the choice of such projection (Theorem \ref{thmfn}).

\begin{parrafo}\label{par12}{\bf Slope of a monic polynomial}. 
Fix a hypersurface $X$ embedded in a smooth scheme $V^{(d)}$ and a closed  $n$-fold point  $x$ of $X$ (i.e., a point of multiplicity $n$). After suitable restriction, in \'etale topology, a smooth morphism $V^{(d)}\overset{\beta}{\longrightarrow} V^{(d-1)}$ can be defined so that $X$ is expressed by a monic polynomial of degree $n$, of the form 
\begin{equation}\label{eq3moni}
f_n(z)=z^n+a_1z^{n-1}+\dots+a_n\in\calo_{V^{(d-1)}}[z];
\end{equation}
where $z$ is a global section of $\calo_{V^{(d)}}$ and $z=0$ is a $\beta$-section. We abuse notation and say that $z$ is a transversal section of $\beta$.
This says that, after suitable restriction in \'etale topology, both at $x$ and $\beta(x)$, the restricted map $\beta|_{X}$:
\begin{equation}\label{eqtri}
\xymatrix@R=2.1pc{
X\ar@{^{(}->}[r]\ar[rd]_{\ \ \ \ \beta|_X}& V^{(d)}\ar[d]^{\beta}\\
 & V^{(d-1)}
}
\end{equation}
is finite. So locally at any $y\in V^{(d-1)}$ we may view $X$ as defined by the monic polynomial  in (\ref{eq3moni}). In particular, the fiber over $y$, say $\beta^{-1}(y)$, is given by $\overline{f}_n(z)=z^n+\bar{a}_1z^{n-1}+\dots+\bar{a}_n\in k(y)[z]$, where $k(y)$ is the residue field at the point.

Let $f_nW^n$ denote the Rees algebra $\calo_{V^{(d)}}[f_n(z)W^n]$. Note that $\Sing(f_nW^n)$ is the set of $n$-fold points of the hypersurface. Here we assume that $\Sing(f_nW^n)$ has no components of codimension one in $V^{(d)}$ (namely that $f_n(z)\neq z^n$ for any expression as (\ref{eq3moni})).

\begin{definition}\label{dfsl}
Assume, as before, that $X$ is defined by the monic polynomial in (\ref{eq3moni}). The \emph{slope of $f_n(z)W^n$ at $y\in V^{(d-1)}$} is  the rational number
\begin{equation}\label{slopef}
Sl(f_n(z)W^n)(y)=\min_{1\leq j\leq n}\Big\{\frac{\nu_y(a_j)}{j}\Big\}.
\end{equation}
\end{definition}

Geometrically, the slope defined by the equation (\ref{eq3moni}) is the biggest rational number $q$ so that all pairs $(\nu_y(a_j),n-j)$ lie above the line through $(0,n)$ and $(nq,0)$:

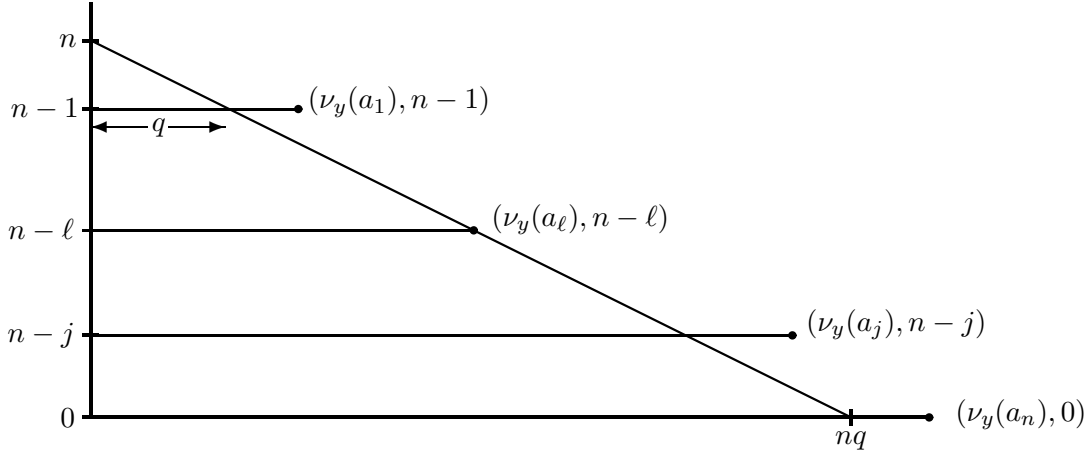
\begin{center}
\setlength{\unitlength}{1mm}
\begin{picture}(130,60)
\allinethickness{1pt}

\put(10,0){\line(0,1){55}}
\put(10,0){\line(1,0){110}}
\put(10,50){\line(2,-1){100}}

\put(9,49.9){\line(1,0){2}}
\put(7,49.9){\makebox(0,0){$n$}}

\put(9,40.9){\line(1,0){2}}
\put(4,40.9){\makebox(0,0){$n-1$}}

\put(9,24.8){\line(1,0){2}}
\put(3.5,24.8){\makebox(0,0){$n-\ell$}}

\put(9,10.9){\line(1,0){2}}
\put(3.5,10.9){\makebox(0,0){$n-j$}}

\put(9,0){\line(1,0){2}}
\put(7,0){\makebox(0,0){$0$}}

\put(110,-1){\line(0,1){2}}
\put(110,-3){\makebox(0,0){$nq$}}

\allinethickness{0.6pt}
\put(9,40.9){\line(1,0){28.3}}
\put(9,10.9){\line(1,0){93.3}}
\put(9,24.8){\line(1,0){51.3}}

\allinethickness{1pt}
\put(37.3,40.9){\circle*{0.8}}
\put(50.6,41.9){\makebox(0,0){$(\nu_y(a_1),n-1)$}}

\put(60.4,24.8){\circle*{0.8}}
\put(74.4,26.3){\makebox(0,0){$(\nu_y(a_\ell),n-\ell)$}}
\put(102.3,10.9){\circle*{0.8}}
\put(116,12.4){\makebox(0,0){$(\nu_y(a_j),n-j)$}}
\put(120.3,0){\circle*{0.8}}
\put(132.4,0.4){\makebox(0,0){$(\nu_y(a_n),0)$}}

\allinethickness{0.6pt}
\put(17.5,38.5){\vector(-1,0){7.5}}
\put(19,38.5){\makebox(0,0){$q$}}
\put(20.2,38.5){\vector(1,0){7.5}}

\end{picture}

\vspace{0.3cm}
\end{center}

Changes on the variable $z$ imply changes on (\ref{eq3moni}), and hence on the value (\ref{slopef}). We aim to find the biggest possible value of the slope at the fixed point $y$. A first step in this direction will be address in Remark \ref{rmk44pe}.

 The following technical lemma and remark gather some results to be used in this and further sections, needed to establish the optimality of the rational number.
 
\begin{lemma}(Zariski's Multiplicity Lemma)\label{ZML}. Let $f_n$ be as in (\ref{eq3moni}), defining $X$ as in (\ref{eqtri}).

(1) Fix $y\in\beta(\Sing(f_nW^n))$. Then $\bar{f}_n(z)=(z-\overline{\alpha})^n\in k(y)[z]$, for a suitable  $\overline{\alpha}\in k(y)$, the residue field at $\calo_{V^{(d-1)},y}$. In particular,  the fiber $\beta^{-1}(y)$ has a unique point which is rational.

(2) Assume that $C\subset\Sing(f_nW^n)$ is smooth and irreducible. $\beta(C)$ is smooth and moreover $\beta|_C$ induces an isomorphism $\beta|_C:C\overset{\cong}{\longrightarrow}\beta(C)$ (see Proposition \ref{elim} (3)). In particular, if $V^{(d-1)}$ is affine, there is a global section $\alpha\in\calo_{V^{(d-1)}}$ so that $(z-\alpha)\in I(C)\subset\calo_{V^{(d)}}$.

(3) Suppose that $C'$ is a smooth and irreducible component of $\beta(\Sing(f_nW^n))$. There is a unique irreducible component $C\subset\Sing(f_nW^n)$ such that $\beta(C)=C'$. Moreover, $C$ and $C'$ are isomorphic and hence $C$ is smooth.

\end{lemma}
Further details can be found in \cite{VV4}, \cite{VKyoto} and \cite{BV3}. Note that (1) says that 
$\beta|_{\Sing(f_nW^n)}:\Sing(f_nW^n)\overset{}{\longrightarrow}\beta(\Sing(f_nW^n))$ is a set theoretical bijection, and that corresponding points have the same residue field. This is a byproduct of Zariski multiplicity Lemma. (2) and the first part of (3) follow essentially from this fact.
As for the second half of (3), note that $\beta|_C:C\overset{}{\longrightarrow}\beta(C)=C'$ is a finite birational morphism, and that $C'$ is normal.

\begin{remark}\label{rmkZ} 
\ \ (1) Fix $y\in V^{(d-1)}$. Let $\bar{f}_n(z)=z^{n}+\bar{a}_1z^{n-1}+\dots+\bar{a}_n\in k(y)[z]$ be the class of $f_n(z)$ on the fiber. Then, $Sl(f_n(z)W^n)(y)=0$ if and only if $\bar{f}_n(z)\not=z^n$.

When $y\in\beta(\Sing(f_nW^n))$, $\overline{f}_n(z)=(z-\overline{\alpha})^n$ and a change of the form $z_1=z+\gamma$ (suitable $\gamma\in\calo_{V^{(d-1)},y}$) can be defined, with the property that $\bar{\gamma}=\bar{\alpha}$ in $k(y)$. The restriction to the fiber over of $y$ is $\bar{f}'_n(z_1)=z_1^n$, and hence $Sl(f_n'(z_1))(y)>0$. 
Moreover, we claim that in this case $Sl(f'_n(z_1))(y)\geq 1$. To check this, notice that the condition $Sl(f_n'(z_1))(y)>0$ implies that there is a unique point of $X=V(f_n'(z_1))$, say $y'\in V^{(d)}$, dominating $y$. If $\{x_1,\dots, x_\ell\}$ is a regular system of parameters at $\calo_{V^{(d-1)},y}$, then $\{z_1,x_1,\dots, x_\ell\}$ is a regular system of parameters at $\calo_{V^{(d)},y'}$. As $f'_n(z_1)$ has multiplicity $n$ at $\calo_{V^{(d)},y'}$, it follows that $f_n'(z_1)=z_1^n+b_1z^{n-1}+\dots+b_n\in\id{z_1,x_1,\dots,x_\ell}^n$, and hence each $b_i\in\id{x_1,\dots,x_\ell}^i$.

\vspace{0.15cm}

(2) Fix  $x\in C\subset \Sing(f_nW^n)$ and set $\x=\beta(x)$. One can argue as in Lema \ref{ZML}, (2), to show that 
there is a change of the form $z_1=z-\alpha$ with $\alpha\in\calo_{V^{(d-1)},\x}$ so that $Sl(f'_n(z_1)W^n)(y)>0$ where $y$ denotes the generic point of $\beta(C)$. In particular, $Sl(f'_n(z_1)W^n)(y)\geq 1$.
 
\vspace{0.15cm}

(3) The previous discussion applies if we fix $x\in C'\subset\beta(\Sing(f_nW^n))$, with $C'$ smooth and irreducible. In fact,  a change of variables of the form $z_1=z+\alpha$ with $\alpha\in\calo_{C',x}$ can be considered, so that $Sl(f'_n(z_1)W^n)(y)>0$ where $y$ denotes the generic point of $C'$. Moreover, applying (1), $Sl(f'_n(z_1)W^n)(y)\geq 1$
\end{remark}

We want to give a criterium to know when $Sl(f_n(z)W^n)(y)$ is optimal. 

\begin{definition}
Fix a transversal parameter $z$ so that $f_n(z)=z^{n}+a_1z^{n-1}+\dots+a_n\in\calo_{V^{(d-1)}}[z]$ defines the hypersurface $X$ after a suitable restriction. Set $q=Sl(f_n(z)W^n)(y)$. Let $r_j=\nu_y(a_j)$ ($j=1,\dots,n$) denote the order of each coefficient at $\calo_{V^{(d-1)},y}$. Fix a regular system of coordinates at $\calo_{V^{(d-1)},y}$, say $\{x_1,\dots,x_{\ell-1}\}$. At the completion, say $\widehat{\calo}_{V^{(d-1)},y}=k(y)[[x_1,\dots,x_{\ell-1}]]$, define $f_{n}(z)=z^n+\widehat{a}_1z^{n-1}+\dots+\widehat{a}_n$, with 
$$\widehat{a}_j=\sum_{i\geq r_j}A_j^{i}\in k(y)[[x_1,\dots,x_{\ell-1}]],$$
where $A_j^i$ is homogeneous of degree $i$ in the variables $x_1,\dots,x_{\ell-1}$. Here $\ell=d-1$ if $y$ is a closed point.

The \emph{weighted initial form} of $f_n(z)$ is defined as
\begin{equation}\label{eqaj}
\win_y(f_n(z))=\sum_{0\leq j\leq n}A_j^{jq}Z^{n-j}\in gr(\calo_{V^{(d-1)},y})[Z],
\end{equation}
where each $A_j^{jq}$ is weighted homogeneous of degree $jq$, and where $A_j^{jq}=0$ if $jq\not\in\mathbb{Z}$.

Note that $\win_y(f_n(z))$ is defined in $gr(\calo_{V^{(d-1)},y})[Z]$, and it is weighted homogeneous of degree $nq$ if $Z$ is given weight $q$ and each $X_i=In(x_i)$ has weight $1$ for $i=1,\dots,d-1$. To ease notation we take now $z$ with weight $q$ and each $x_i$ with weight $1$ at the ring of formal power series.
\end{definition}

\begin{remark}\label{win0}
Suppose that $Sl(f_n(z)W^n)(y)=0$. Then, $\win_y(f_n(z))$ is weighted homogeneous of degree $0$ where $z$ is endowed with weight $q=0$ and each $x_i$ with weight $1$. In such case, $\win_y(f_n(z))$ is defined in $k(y)[Z]$ and there is a natural identification of this polynomial with the equation defining the fiber of $X$ over $y$, namely $\win_y(f_n(z))=\bar{f}_n(Z)\in k(y)[Z]$, where $\bar{f}_n(z)$ is the equation that defines  the fiber.
\end{remark}

\begin{remark}\label{rmk44pe}
Note that in (\ref{eqaj}) $A_0^0=1$, and hence $\win_y(f_n(z))$ is a monic polynomial of degree $n$. Moreover, note also that $\win_y(f_n(z))\not=Z^n$  by definition.

One can check that, $\win_y(f_n(z))$ is an $n$-th power if and only if $\win_y(f_n(z))=(Z+A)^n$ for some $A\in gr(\calo_{V^{(d-1)},y})$,  which must be homogeneous of degree $q$. If this occurs, then $q\in\mathbb{Z}$, and hence there is an element $\alpha\in\calo_{V^{(d-1)},y}$ so that $In_y(\alpha)=A$. The change of variables $z_1= z+\alpha$ gives rise to a strictly higher slope, i.e., $Sl(f_n'(z_1)W^n)(y)>Sl(f_n(z)W^n)(y)$.

The previous discussion shows that a change of the form $z_1= z+\alpha$ can increase the slope if and only if $\win_y(f_n(z))$ is an $n$-th power. 

As $\Sing(f_nW^n)$ has no components of codimension $1$ in $V^{(d)}$, by assumption, after finitely many changes of the variable $z$ as above, we may assume that $\win_y(f_n(z))$ is not an $n$-th power.
\end{remark}

\begin{definition}\label{def:q}
A monic polynomial $f_n(z)$ is said to be in \emph{normal form} at a point $y\in V^{(d-1)}$ if the weighted initial form $\win_y(f_n(z)W^n)$ is not an $n$-th power at $gr(\calo_{V^{(d-1)},y})[Z]$. 
\end{definition}

\begin{proposition}\label{por71}
Fix a point $y\in V^{(d-1)}$ and a polynomial $f_n(z)=z^n+a_1z^{n-1}+\dots+a_n\in\calo_{V^{(d-1)}}[z]$ which is in normal form at $y$, i.e., assume that $z$ is such that $\win_y(f_n(z))$ is not an $n$-th power at $gr(\calo_{V^{(d-1)},y})[Z]$.
Then $y\in\beta(\Sing(f_nW^n))$ if and only if $q=Sl(f_n(z)W^n)(y)\geq 1$.
\end{proposition}

\begin{proof} 
If $y\in\beta(\Sing(f_nW^n))$, and since $\win_y(f_n(z))$ is not an $n$-th power, then $Sl(f_n(z)W^n)(y)>0$. Under these conditions Remark \ref{rmkZ} (1) says that $Sl(f_n(z)W^n)(y)\geq 1$.

Conversely, suppose that $Sl(f_n(z)W^n)(y)\geq 1$. Lemma \ref{rmkZ} (1) says that $Sl(f_n(z)W^n)(y)> 0$ implies that there is a unique point, say $y'\in X=V(f_n(z))$, dominating $y$. Fix a regular system of parameters $\{x_1,\dots,x_\ell\}$ at $\calo_{V^{(d-1)},y}$, and recall that $\{z,x_1,\dots,x_\ell\}$ is a regular system of parameters  at $\calo_{V^{(d)},y'}$. Finally, $Sl(f_n(z)W^n)(y)\geq 1$ implies that $\nu_{y}(a_i)\geq i$ and hence $y'$ is an $n$-fold point of $X$.
\end{proof}

\begin{remark}
Let the assumptions be as in (\ref{eqtri}), where $f_n(z)=z^n+a_1z^{n-1}+\dots+a_n\in\calo_{V^{(d-1},y}[z]$. A function, say $q_\beta:\Sing(f_nW^n)\longrightarrow \mathbb{Q}_{>0}$, will be defined by setting:
$$q_\beta(x):=\max_{z_1}\{Sl(f_n(z_1)W^n)(\beta(x))\},$$
where $z_1=z+\alpha$ for all the possible $\alpha\in\calo_{V^{(d-1)},\beta(x)}$ (see Remark \ref{rmkZ} (1)).
Note that, if $f_n(z)$ is in normal form at $\beta(x)$, then  Remark \ref{rmk44pe} ensures that
$$q_\beta(x)=Sl(f_n(z)W^n)(\beta(x)).$$
\end{remark}

\begin{theorem}\label{thmfn}
Fix a projection $V^{(d)\overset{\beta}{\longrightarrow}}V^{(d-1)}$ together with a monic polynomial $f_n(z)=z^n+a_1z^{n-1}+\dots+a_n\in\calo_{V^{(d-1)}}[z]$, where $\{z=0\}$ is a section of $\beta$. Consider a point $x\in\Sing(f_nW^n)$ and assume that $f_n(z)$ is in  normal form at $\beta(x)$. 
Let  $q$ denote the slope of $f_n(z)$ at $\beta(x)$. The rational number $q$ is completely characterized by the weak equivalence class of the algebra $\G=\calo_{V^{(d)}}[f_nW^n]$ in a neighborhood of $x$.
\end{theorem}

\begin{corollary}
Fix a hypersurface $X\subset V^{(d)}$ of maximum multiplicity $n$ and two projections $\beta:V^{(d)}\longrightarrow V^{(d-1)}$ and $\beta':V^{(d)}\longrightarrow V'^{(d-1)}$, each is the setting of (\ref{eqtri}). For any $n$-fold point $x\in X$:
$$q_\beta(x)=q_{\beta'}(x).$$
\end{corollary}

\begin{proof}
The rational number $q_{\beta}(x)$ is completely determined in terms of the weak equivalence class of $\G=\calo_{V^{(d)}}[f_nW^n]$ at $x$, and this is independent of the chosen projection.
\end{proof}
\end{parrafo}

\begin{proofTHM} 
Set $\x=\beta(x)$. 
As $x\in\Sing(f_nW^n)$, Lemma \ref{ZML} (1) says that $\overline{f}_n(z)=(z-\overline{\alpha})^n$ (with $\alpha\in k(\x)$). Here $\overline{f}_n(z)$ denotes the restriction to the fiber over $\x$. Since $f_n(z)$ is a normal form at $\x$, then Remark \ref{rmkZ} (1) ensures that $\overline{\alpha}=0$ and hence $z$ vanishes at $x$. A regular system of parameters $\{x_1,\dots,x_e\}$ at $\calo_{V^{(d-1)},\x}$ can be extended to $\{z,x_1,\dots,x_{e}\}$ so as to be a regular system of parameters at $\calo_{V^{(d)},x}$. Note also that $x\in\Sing(f_nW^n)$, so $\nu_{\x}(a_j)\geq j$ for $j=1,\dots,n$.

Express the monic polynomial as $f_n(z)=z^n+\widehat{a}_1z^{n-1}+\dots+\widehat{a}_n$ at the completion $\widehat{\calo}_{V^{(d)},x}$, and set
$$\widehat{a}_j=A_j^{jq}+\widetilde{A}_j\in k(\x)[[x_1,\dots,x_{e}]]$$
where $A_j^{jq}$ is homogeneous of degree $jq$ and $\widetilde{A}_j$ has order $>jq$. Here $A_j^{jq}=0$ if $jq\not\in\mathbb{Z}$.
Recall that $f_n(z)$ is in normal form, i.e., that
\begin{equation}\label{eqthm}
\win_\x(f_n(z))=Z^n+\sum_{j=1}^{n}A_j^{jq}Z^{n-j}\in gr_\x(\calo_{V^{(d-1)}})[Z]
\end{equation}
is not an $n$-th power. Denote by $r_j=\nu_\x(a_j)$ ($j=1,\dots,n$) the order of each coefficient at $\calo_{V^{(d-1)},\x}$. Note that $A_j^{jq}=0$ if $\frac{r_j}{j}>q$.

\vspace{0.2cm}

\noindent $\bullet$ {\bf Stage A:}
Let $V^{(d)}\times \mathbb{A}^1$ denote the product of $V^{(d)}$ with the affine line. Projection on first coordinate enable us to take the pull-back of $f_n(z)$, in a neighborhood of $(x,0)\in V^{(d)}\times\mathbb{A}^1$. The natural extension 
\begin{equation}\label{betaid}
\xymatrix@R=0pc@C=3pc{
V^{(d)}\times\mathbb{A}^1\ar[r]^{\beta\times id} & V^{(d-1)}\times\mathbb{A}^1
}\end{equation} 
is a smooth morphism that maps $(x,0)$ to $(\x,0)$.The natural identification of $\win_{(\x,0)}(f_n(z))$ with $\win_\x(f_n(z))$ guarantees that $\win_{(\x,0)}(f_n(z))$ is not an $n$-th power. This identification together with Remark \ref{rmkZ} (1) ensure that
$$Sl(f_n(z)W^n)((\x,0))=Sl(f_n(z)W^n)(\x)=q\geq1.$$

Fix coordinates $\{z,x_1,\dots,x_{e},t\}$ locally at $(x,0)$, here $\{z,x_1,\dots,x_{e}\}$ is the regular system of parameters at $\calo_{V^{(d)},x}$ mentioned before. Consider the monoidal transformation with center $p_0=(x,0)$ and let $p_1$ be the intersection of the new exceptional hypersurface, say $H_1$, with the strict transform of  $x\times\mathbb{A}^1$.

The point $p_1$ can be identified with the origin of the $U_t$-chart, ($U_t=\Spec(k[\frac{z}{t},\dots,\frac{x_{e}}{t},t])$). 
Now set 
$$f_n^{(1)}(z_1)=z_1^n+t^{r_1-1}\widetilde{a}_1^{(1)}z_1^{n-1}+\dots+t^{r_n-n}\widetilde{a}_n^{(1)},$$
at the completion of the local ring of at $p_1$, and check that:
$$t^{r_j-j}\widetilde{a}_j^{(1)}z_1^{n-j}=t^{r_j-j}\big(A_j^{jq}+t^{\gamma_j}\widetilde{A}_j\big)z_1^{n-j}$$
with $\gamma_j>0$ for $j=1,\dots,n$. Here $z_1=\frac{z}{t}$ and $\widetilde{a}^{(1)}_j$ is the strict transform of $\widehat{a}_j$.

This process can be iterated $N$-times, defining a sequence of monoidal transformations at $p_1,\dots$, $p_{N-1}$, where each $p_j$ is the point of intersection of the new exceptional component, say $H_j$, with the strict transform of $x\times\A^1$.

The final strict transform of $f_n(z)$ at the $U_t$-chart is given by
\begin{equation}\label{fN}
f_n^{(N)}(z_N)=z_N^n+t^{N(r_1-1)}\widetilde{a}_1^{(N)}z_N^{n-1}+\dots+t^{N(r_n-n)}\widetilde{a}_n^{(N)},
\end{equation}
where
$$t^{N(r_j-j)}\widetilde{a}_j^{(N)}z_N^{n-j}=t^{N(r_j-j)}\big(A_j^{jq}+t^{\gamma'_j}\widetilde{A}_j\big)z_N^{n-j}$$
with $\gamma'_j>0$ for $j=1,\dots,n$.

It may occur that $\{f_n^{(N)}=0\}\cap H_N$ is a $2$-codimensional component of the $n$-fold points of $\{f_n^{(N)}=0\}$. In that case, we will  fix $N$ sufficiently large, and we look for the largest number of successive monoidal transformations that can be defined with center in codimension $2$. We explain bellow how these centers are to be chosen.

We also show that the rational number $q$ is characterized in terms of these monoidal transformations:

\vspace{0.2cm}

\noindent $\bullet$ {\bf Stage B}: Firstly, consider a monoidal transformation along the center $\id{z_N,t}$, if possible. Denote $\frac{z_N}{t}$ by $z_{N+1}$  at the $U_t$-chart. The transform is
$$f_n^{(N+1)}(z)=z^n_{N+1}+t^{N(r_1-1)-1}\widetilde{a}_1^{(N+1)}z^{n-1}_{N+1}+\dots+t^{N(r_n-n)}\widetilde{a}_{n}^{(N+1)},$$
where
$$t^{N(r_j-j)-j}\widetilde{a}_j^{(N+1)}z_{N+1}^{n-j}=t^{N(r_j-j)-j}\big(A_j^{jq}+t^{\gamma''_j}\widetilde{A}_j\big)z^{n-j}_{N+1}$$
with $\gamma''_j>0$ for $j=1,\dots,n$.

After $\ell$ monoidal transformations along centers of codimension $2$ of the form  $\id{z_{N+i},t}$, the exponents of $t$ in each coefficient is $N(r_j-j)-\ell j$. Therefore $\id{z_{N+\ell},t}$ is a permissible center whenever $N(r_j-j)-\ell j\geq j$ for all $j\in\{1,\dots,n\}$. In particular, this condition requires that 
$$\ell\leq\min_{1\leq j\leq n}\Big\{N\big(\frac{r_j}{j}-1\big)-1\Big\}=N(q-1)-1.$$

The geometric interpretation of the previous sequence of $\ell$-monoidal transformation can be described as follows: Set $X_N=\{f_n^{(N)}=0\}$ (see (\ref{fN})) and let $H_N$ denote the exceptional hypersurface $t=0$. The sequence previously constructed can be expressed in terms of the diagram
\begin{equation}\label{seqXN}
\begin{xymatrix}@R=0.1cm@C=1cm{
X_N & X_{N+1} &  & X_{N+\ell}\\
V^{(d+1)}_N & V_{N+1}^{(d+1)}\ar[l]_{\pi_{N+1}} &\cdots\ar[l]_{\pi_{N+2}} & V_{N+\ell}^{(d+1)}.\ar[l]_{\pi_{N+\ell}}
}\end{xymatrix}
\end{equation}
If $H_{N+i+1}$ denotes the exceptional hypersurface of $\pi_{N+i}$, then the centers of this monoidal transformations are defined by $X_{N+i+1}\cap H_{N+i+1}$.

Let $\widetilde{\beta}$ denote the morphism $V^{(d)}\times \mathbb{A}^1\longrightarrow V^{(d-1)}\times \mathbb{A}^1$ in (\ref{betaid}). This morphism has a natural lifting to the sequence  of monoidal transformations of length $N$ in Stage A), and also to the sequence (\ref{seqXN}). This is guaranteed by Proposition \ref{elim} (3). In particular, for each index $i=0,\dots,\ell$, the previous sequence defines morphisms 
\begin{equation}\label{eqbetai}
\widetilde{\beta}_{N+i}:V^{(d+1)}_{N+i}\longrightarrow V^{(d)}_{N+i}.
\end{equation}

Set 
\begin{equation}\label{eqmt}
\ell_N=\lfloor N(q-1)-1\rfloor.
\end{equation}
We finally claim that if $\ell=\ell_N$, then $X_{N+\ell}\cap H_{N+\ell}$ is not a permissible center for $X_{N+\ell}$ (see (\ref{seqXN})).

In fact, whenever $Nq\not\in\mathbb{Z}$, and after applying $\ell_N$ monoidal transformations along these centers of codimension $2$, one gets $V(\id{z_{N+\ell_N},t})\subset X_{N+\ell_N}$ and $0<Sl(f_n^{(N+\ell_N)}W^n)(\xi_H)<1$. So $H_{N+\ell_N}=\{t=0\}$ cannot be a component of $\widetilde{\beta}_{N+\ell}(\Sing(f_n^{(N+\ell_N)}W^n))$ (see Proposition \ref{por71}). 

Note that if $Nq\in\mathbb{Z}$, then $N(q-1)-1$ is a positive integer. In this case, after applying $N(q-1)-1$ monoidal transformations along these 2-codimensional centers, the final strict transform of $f_n^{(N)}(z)$ is given by
$$f_n^{(N+\ell_N)}(z)=z_{N+\ell_N}^n+t^{N(r_1-q)}\widetilde{a}_1^{(N+\ell_N)}z^{n-1}_{N+\ell_N}+\dots+t^{N(r_n-qn)}\widetilde{a}_{n}^{(N+\ell_N)}$$
at the chart of interest. We claim that $X_{N+\ell_N}\cap H_{N+\ell_N}$ is not a permissible center.

Note that $Sl(f_n^{(N+\ell_N)}(z_{N+\ell_N})W^n)(\xi_H)=0$. Remark \ref{win0}, shows that $\win_{\xi_H}(f_n^{(N+\ell_N)}(z_{N+\ell_N}))$ can be identified with the equation defining the fiber over $\xi_H$:
$$f^{(N+\ell_N)}_n(z_{N+\ell_N})|_{t=0}=z_{N+\ell_N}^n+\sum_{j=1}^{n}A_j^{jq}z^{n-j}_{N+\ell_N}.$$
The expression of the right hand side can be naturally identified with $\win_\x(f_n(z))$ (see (\ref{eqthm})). As we assume that $\win_\x(f_n(z))$ is not an $n$-th power, Proposition \ref{por71} together with Remark \ref{rmkZ} (3) ensure that $H_{N+\ell_N}$ is not a component of $\widetilde{\beta}_{N+\ell_N}(\Sing(f_n^{(N+\ell_N)}W^n))$.

This discussion will show that $q$ is totally characterized by Hironaka's weak equivalence class of the $n$-fold points of the hypersurface $\{f_n=0\}$. This point will be further clarify in Remark \ref{513}. Notice that the construction of the sequences in Stage A) and Stage B), together with (\ref{eqmt}) lead to the equality
\begin{equation}\label{eqlim}
\lim_{N\rightarrow \infty}\frac{\ell_N}{N}=q-1.
\end{equation}
\end{proofTHM}

\begin{remark}\label{513}
We claim that the rational
number $q$ can be expressed in terms of Hironaka's weak equivalence class. Recall here that the weak equivalence class of $f_nW^n$ is defined in the context of $k$-algebras of finite type, whereas  the local rings $\calo_{V^{(d)},x}$ are not within this class.

The claim is straightforward when the point $x$ is a closed point. But it requires some clarification if $x\in\Sing(f_nW^n)$ is not a closed point. Let $Y$ denote the variety with generic point $x$. As the weak equivalence class allows restriction to open sets, we may assume that $Y$ is smooth.

Fix a closed point $p$ at $Y$ and fix local coordinates $\{z,x_2,\dots,x_d\}$ at $\calo_{V^{(d)},p}$, so that $\{x_2,\dots,x_d\}$ is a regular system of parameters at $\calo_{V^{(d-1)},\beta(p)}$. We may assume, applying Lemma \ref{ZML} (2), that
\begin{enumerate}
\item $I(Y)=\id{z,x_2,\dots,x_\ell}$,
\item $f=z^{n}+a_1z^{n-1}+\dots+a_n$,
\item $I(\beta(Y))=\id{x_2,\dots,x_\ell}$.
\end{enumerate}
In particular, $gr_{I(\beta(Y))}(\calo_{V^{(d-1)},\beta(p)})=\calo_{\beta(Y),\beta(p)}[X_2,\dots,X_\ell]$ and $k(\x)$ is the quotient field of $\calo_{\beta(Y),\beta(p)}$ (for $\x=\beta(x)$). Note that $\win_Y(f_n)\in gr_{I(\beta(Y))}(\calo_{V^{(d-1)},\beta(p)})$ can be naturally identified with $\win_\x(f_n)$ via the inclusion $gr_{I(\beta(Y))}(\calo_{V^{(d-1)},\beta(p)})\subset gr_{k(\x)}(\calo_{V^{(d)},x})$. In this setting, one can check that  the sequences of transformations used in the previous proof to determinate the rational number $q$ are expressed in terms of the weak equivalence class of $f_nW^n$. This proves the claim for the case in which $x$ is a non-closed point.
\end{remark}

\section{Slope of a Rees algebra and the $d-1$-dimensional H-function ($\tau\geq 1$)}\label{sec6}
\begin{parrafo}

We rephrase the results and invariants discussed in Section \ref{sec5} but now in the context of  Rees algebras. Here Theorem \ref{thmqG} parallels Theorem \ref{thmfn} in the previous section.

Throughout the section we fix a Rees algebra $\G=\bigoplus I_nW^n$ over $V^{(d)}$ and a assume that $\tau_\G\geq 1$ along $\Sing(\G)$. Consider a transversal projection $V^{(d)}\overset{\beta}{\longrightarrow}V^{(d-1)}$, and assume in addition that $\G$ is $\beta$-differential (see \ref{rp23}). 

\begin{proposition}\label{pr_local}{\bf (Local presentation)}. Fix a closed point $x\in\Sing(\G)$ for wich $\G$ is simple, and a locally defined $V^{(d)}\overset{\beta}{\longrightarrow} V^{(d-1)}$, transversal to $\G$ at $x$. Assume that $\G$ is a $\beta$-relative differential Rees  algebra. Fix  $f_nW^n\in\G$ so that $f_n$ has order $n$ at  $\calo_{V^{(d)},x}$. Assume that $f_n=f_n(z)$ is a monic polynomial of degree $n$ in $\calo_{V^{(d-1)},\beta(x)}[z]$, where $z$ is a $\beta$-section.  Then, in a neighborhood of $x$, $\G$ has the same integral closure as
\begin{equation}\label{eqlocalpresen}
\calo_{V^{(d)}}[f_n(z)W^n,\Delta^{j}_z(f_n(z))W^{n-j}]_{1\leq j\leq n-1}\odot\beta^*(\R_{\G,\beta}),
\end{equation}
where $\Delta^{j}_z$ are suitable $\beta$-differential operators of order $j$.
Moreover, $\R_{\G,\beta}$ is non-zero whenever $\Sing(\G)$ is not  of codimension one in $V^{(d)}$.
\end{proposition}
	
\begin{proof} \cite[Proposition 2.11.]{BeV1}.  
\end{proof}

\begin{remark}\label{rmk43}
The differential operators $\Delta^j_z$ in the previous proposition are the operators obtained by the Taylor morphism:
This is a morphism of $S$-algebras, say $Tay: S[Z]\longrightarrow S[Z, T]$,
 defined by setting $Tay(Z)=Z+T$ (Taylor expansion). Here
$$Tay(f(Z))= f(Z+T)= \sum \Delta^{r}(f(Z))T^r,$$
and these operators $ \Delta^{r}\!: S[Z] \longrightarrow S[Z]$ are defined by this morphism. It is well known that $\{\Delta^{0}, \Delta^{1}, \dots,\Delta^{r}\}$ is a basis of the free module of $S$-differential operators of order $r$. The same applies here for $\calo_{V^{(d-1)}} [z] $: the set $\{\Delta^{0}_z, \Delta^{1}_z, \dots,\Delta^{r}_z\}$ spans the sheaf of differential operators of order $r$ relative to the smooth morphism $\beta:V^{(d)}\longrightarrow V^{(d-1)}$. Moreover, as $V^{(d)}$ is \'etale over $V^{(d-1)}\times \mathbb A^1$, the previous set also generated $Diff^r_{\beta}$ ($\beta$-linear differential operators of order $r$).

\end{remark}

\begin{definition}\label{pr_locp}({\bf Presentations}).
Fix, after suitable restriction in \'etale topology, a projection $V^{(d)}\overset{\beta}{\longrightarrow} V^{(d-1)}$ transversal to a simple $\beta$-relative differential Rees algebra $\G$. Assume that $\Sing(\G)$ has no components of codimension one and that
\begin{enumerate}
\item[i)] There is a $\beta$-section $z$ (or global section so that $\{dz\}$ is a basis of the locally free module of $\beta$-differentials, say $\Omega_\beta^1$).

\item[ii)] There is an element $f_{n}(z)W^{n}\in \G$, where $f_{n}(z)$ is a monic polynomial of order $n$, say
$$f_{n}(z)=z^{n}+a_1z^{n-1}+\dots+a_{n}\in \calo_{V^{(d-1)}}[z],$$
where each $a_i$ is a global function on $V^{(d-1)}$. 
\end{enumerate}

In this case Proposition \ref{pr_local} holds, namely $\G$ has the same integral closure as
\begin{equation} \label{eqdpl} 
\calo_{V^{(d)}}[f_{n}(z)W^{n},\Delta_z^{j}(f_{n}(z))W^{n-j}]_{1\leq j\leq n-1}\odot\beta^*(\R_{\G,\beta}).
\end{equation}

We say that $\beta:V^{(d)}\longrightarrow V^{(d-1)}$, the $\beta$-section $z$,  and $f_{n}(z)=z^{n}+a_1z^{n-1}+\dots+a_{n}$ define a \emph{presentation} of $\G$. These data will be denoted by:
\begin{equation}\label{eqld}
\P(\beta:V^{(d)}\longrightarrow V^{(d-1)},  z,  f_{n}(z)=z^{n}+a_1z^{n-1}+\dots+a_{n}),
\end{equation}
or simply by $\P(\beta,  z,  f_{n}(z))$. 
\end{definition}

\begin{definition}
Fix a Rees algebra $\G$ and a presentation $\P(\beta,z,f_n(z))$. Define the \emph{slope of $\G$ relative to $\P$ at a point $y\in V^{(d-1)}$} as
$$Sl(\P)(y):=\min_{1\leq j\leq n}\Big\{\frac{\nu_y(a_j)}{j},\ord(\R_{\G,\beta})(y)\Big\}=\min\{Sl(f_n(z)W^n)(y),\ord(\R_{\G,\beta})(y)\}.$$
\end{definition}

\begin{remark}
A change of variables of the form $z_1= z+\alpha$, $\alpha\in\calo_{V^{(d-1)}}$ (a global section of $V^{(d-1)}$), gives rise to a new presentation defined in a natural way, say $\P_1=\P_1(\beta,z_1,f_n'(z_1))$.

Assume first that $Sl(\P)(y)=Sl(f_n(z)W^n)(y)<\ord(\R_{\G,\beta})(y)$. There is a
change of the form $z_1=z+\alpha$, so that the new presentation $\P_1$ has bigger slope at $y$ if and only if $\win_y(f_n(z))$ is an $n$-th power (see Remark \ref{rmk44pe}). 

On the other hand, if  $Sl(\P)(y)= \ord(\R_{\G,\beta})(y)$, the slope at $y$ cannot increase by any change of the presentation of this type.
\end{remark}

\begin{definition}\label{nfG}
Fix a Rees algebra $\G$. A presentation $\P=\P(\beta,z,f_n(z))$  is said to be in \emph{normal form at $y\in V^{(d-1)}$} if one of the following two conditions holds:
\begin{itemize}
\item Either $Sl(\P)(y)=\ord(\R_{\G,\beta})(y)$,
\item or $Sl(\P)(y)= Sl(f_n(z)W^n)(y)<\ord(\R_{\G,\beta})(y)$ and $\win_{y}(f_n(z))$ is not an $n$-th power.
\end{itemize}
\end{definition}

\begin{remark}
Fix, as in Definition \ref{pr_local}, a Rees algebra $\G$, a transversal projection $V^{(d)}\overset{\beta}{\longrightarrow}V^{(d-1)}$ and a point $y\in\beta(\Sing(\G))$. Suppose given two different presentations $\P_1=\P_1(\beta,z,f_n(z))$ and $\P_2=\P_2(\beta, z',g_m(z'))$, both in normal form at $y$. Theorem \ref{thmqG} will show that
$$Sl(\P_1)(y)=Sl(\P_2)(y).$$
Moreover, it will show that this rational number is indeed independent of the chosen projection $\beta$.
\end{remark}

The following result parallels  Proposition \ref{por71}.

\begin{proposition}\label{propHG}
Let $\G$ be a Rees algebra and let $\P=\P(\beta,z,f_n(z))$ be a presentation. Fix a point  $y\in V^{(d-1)}$ and assume that $\P$ is in normal form at $y$.
Then, $y$ is a point of $\beta(\Sing(\G))$ if and only if $Sl(\P)(y)\geq1$.
\end{proposition}

\begin{proof}
Firstly assume that  $Sl(\P)(y)\geq1$. Denote by $x\in V^{(d)}$ the unique point of the fiber $\beta^{-1}(y)$ defined by $z=0$. Fix a regular system of parameters $\{x_1,\dots,x_\ell\}$ at $\calo_{V^{(d-1)},y}$. Since $Sl_y(\P)\geq 1$, then $\nu_y(a_j)\geq j$ (for $j=1,\dots,n$) and $\ord(\R_{\G,\beta})(y)\geq 1$. Set $P=\id{z,x_1,\dots,x_\ell}$, note that $f_n(z)\in P^n$. Check now  that $V(\id{z,x_1,\dots,x_\ell})\subset \Sing(\G)$ (see also Proposition \ref{elim} (2)).

As for the converse, assume that $y\in\beta(\Sing(\G))\ (\subset\Sing(\R_{\G,\beta}))$. In this case, $\ord(\R_{\G,\beta})(y)\geq 1$. Since $\P$ is in normal form at $y$, one of the following two cases can occur:

(1) $Sl(\P)(y)=\ord(\R_{\G,\beta})(y)\geq 1$.

(2)  $Sl(\P)(y)=Sl(f_n(z)W^n)(y)<\ord(\R_{\G,\beta})(y)$ and $\win_y(f_n(z))$ is not an $n$-th power. Then Proposition \ref{por71} ensures that $Sl(\P)(y)\geq 1$.
\end{proof}

\begin{theorem}\label{thmqG}
Fix a Rees algebra $\G$, a point $x\in\Sing(\G)$ so that $\tau_{\G,x}\geq 1$. Consider a presentation $\P=\P(\beta,z,f_n(z))$ which is in normal form at $\beta(x)$.  Then, the rational value $Sl(\P)(\beta(x))$ is completely characterized by the weak equivalence class of $\G$ in a neighborhood of $x$.
\end{theorem}

\begin{proof}
Recall that $Sl(\P)(\x)=\min\{Sl(f_n(z)W^n)(\x),\ord(\R_{\G,\beta})(\x)\}$.

\vspace{0.2cm}

\noindent $\bullet$ Firstly assume that $q:=Sl(f_n(z)W^n)(\x)<\ord(\R_{\G,\beta})(\x)$. In this case, we will argue as in Theorem \ref{thmfn} to prove our claim.  In our coming discussion, we make us of the fact that the transformation law of  $\R_{\G,\beta}$ is that defines for Rees algebras (see (\ref{trRA})).

Fix the same notation as in the proof of Theorem \ref{thmfn} and consider $N$-monoidal transformations at $p_0,p_1,\dots,p_{N-1}$ followed by $\ell_N$ transformations at codimension $2$, where
$$\ell_N=\left\{
\begin{array}{ll}
\lfloor N(q-1)- 1\rfloor & \hbox{if }\ Nq\not\in\mathbb{Z},\\
 N(q-1)- 1 & \hbox{if }\ Nq\in\mathbb{Z}.
\end{array}
\right.$$
We claim that the highest possible number of transformations defined by blowing-up centers of codimension $2$ (in the sense of Stage B)) is exactly this number $\ell_N$ (which is completely characterized by $N$ and $q$). 

The sequence of $N+\ell$ monoidal transformations
$$
\xymatrix@R=0pc@C=2pc{
\G' & \G_1 & & \G_N & \G_{N+1} & & \G_{N+\ell}\\
V^{(d)}\times \mathbb{A}^1 & V^{(d+1)}_1\ar[l]_{\pi_{p_0}}  &\dots\ar[l]_{\pi_{p_1}}  & V^{(d+1)}_N\ar[l]_{\pi_{p_{N-1}}} & V^{(d+1)}_{N+1}\ar[l] & \dots\ar[l] & V^{(d+1)}_{N+\ell}\ar[l]
}$$
(see (\ref{seqintro1})) gives rise to
$$
\xymatrix@R=0pc@C=2pc{
V^{(d-1)}\times \mathbb{A}^1 & V^{(d)}_1\ar[l]  &\dots\ar[l]  & V^{(d)}_N\ar[l] & V^{(d)}_{N+1}\ar[l] & \dots\ar[l] & V^{(d)}_{N+\ell}\ar[l]\\
\R'_{\G,\beta} & (\R_{\G,\beta})_1 & & (\R_{\G,\beta})_{N} & (\R_{\G,\beta})_{N+1} & & (\R_{\G,\beta})_{N+\ell}
}$$ 
together with morphisms $\widetilde{\beta}_i: V^{(d+1)}_i\longrightarrow V^{(d)}_i$ (see (\ref{eqbetai})). In addition, each $(\R_{\G,\beta})_i$ is also defined as the elimination algebra of $\G_i$ via $\beta_i$ (see (3) in (\ref{conmut})). 

Denote by $H$ the exceptional hypersurface introduced  by the last of these transformations. Set $(\R_{\G,\beta})'=(\R_{\G,\beta})_{N+\ell}$. Check that the slope at the generic point of $H$, say $\xi_H$, is 
$$Sl(\P')(\xi_H)=Sl(f'_n(z)W^n)(\xi_H)<\ord((\R_{\G,\beta})')(\xi_H),$$
where $\P'=\P'(\beta',z',f_n'(z'))$ is a presentation given by the transform at $V^{(d+1)}_{N+\ell}$.

In the case $Nq\not \in \mathbb{Z}$, then $0<Sl(\P')(\xi_H)=Sl(f'_n(z')W^n)(\xi_H)<1$, so $H$ is not a component of $\beta'(\Sing(\G'))$ (see Proposition \ref{propHG} and Remark \ref{rmkZ} (3)).

If we assume now that $Nq\in \mathbb{Z}$, then $Sl(\P')(\xi_H)=Sl(f'_n(z')W^n)=0$. Remark \ref{win0} ensures that $\win_{\xi_H}(f_n'(z'))$ can be identified with the equation defining the fiber  over $\xi_H$, i.e. $f'_n(z')|_H$, which can be naturally identified with $\win_\x(f_n(z))$. By hypothesis it is assumed that $\win_\x(f_n(z))$ is not an $n$-th power, so Proposition \ref{propHG} and Remark \ref{rmkZ} (3) applies here to show that $H$ is not a component of $\beta'(\Sing(\G'))$. 

\vspace{0.2cm}

\noindent $\bullet$ To finish the proof, assume now that $Sl(\P)(\x)=\ord(\R_{\G,\beta})(\x)$. Check that $Sl(\P')(\xi_H)=\ord((\R_{\G,\beta})')(\xi_H)<1$ after the transformations indicated before, and hence $H$ is not a component of $\beta'(\Sing(\G'))$.

The previous discussion shows that the value $Sl(\P)(\beta(x))$ is totally characterized by the weak equivalence class of $\G$ as it was done in (\ref{eqlim}).
\end{proof}

\begin{corollary}\label{corbord}
Let $\G$ be a Rees algebra. Fix a point $x\in\Sing(\G)$, so that $\tau_{\G,x}\geq 1$. Consider two different presentations $\P_1=\P(\beta_1,z_1,f_1(z_1))$ and $\P_2=\P(\beta_2,z_2,f_2(z_2))$ which are in normal form at  $\beta_1(x)$ and $\beta_2(x)$, respectively .Then, 
$$Sl(\P_1)(\beta_1(x))=Sl(\P_2)(\beta_2(x)).$$
\end{corollary}

\begin{proof}
Follows straightforward from Theorem \ref{thmqG}.
\end{proof}

The previous discussion leads to the following definition:

\begin{definition}\label{defvord}
Fix a Rees algebra $\G$ and assume that $\tau\geq 1$ along $\Sing(\G)$. A function with rational values, called the  \emph{$d-1$-dimensional H-function}, say
$$\Hord^{(d-1)}(\G)(-):\Sing(\G)\longrightarrow \mathbb{Q}_{\geq 0},$$ 
is defined by setting
\begin{equation}\label{eqdefvord}
\Hord^{(d-1)}(\G)(x):=Sl(\P)(\beta(x))=\min_{1\leq j\leq n}\Big\{\frac{\nu_{\beta(x)}(a_j)}{j},\ord(\R_{\G,\beta})(\beta(x))\Big\},
\end{equation}
at any point $x\in \Sing(\G)$, where $\P=\P(\beta,z,f_n(z)=z^n+a_1z^{n-1}+\dots+a_n)$ is a presentation in normal form at $\beta(x)$.
\end{definition}

Corollary \ref{corbord} (or Theorem \ref{thmqG}) ensures that the previous function is well-defined. Namely, it is intrinsic to $\G$, with independence of $\beta$ and the presentations.

\begin{remark}\

\begin{enumerate}
\item Recall that two Rees algebras $\G$ and $\G'$ with the same integral closure are also weakly equivalent. In particular, $\Sing(\G)=\Sing(\G')$ the H-functions coincide, i.e.,
$$\Hord^{(d-1)}(\G)(x)=\Hord^{(d-1)}(\G')(x)$$
at any point $x\in\Sing(\G)=\Sing(\G')$.

\item Similar statement holds for the Rees algebras $\G$ and $Diff(\G)$ (see \ref{rp23}). In fact, both are weakly equivalent; so at any point $x\in\Sing(\G)=\Sing(Diff(\G))$:
$$\Hord^{(d-1)}(\G)(x)=\Hord^{(d-1)}(Diff(\G))(x).$$
\end{enumerate}
\end{remark}
\end{parrafo}

\section{Simplified presentations and the $d-r$-dimensional H-functions ($\tau\geq r$)}\label{sec7}

\begin{parrafo}

Let $\G$ be a differential Rees  algebra in $V^{(d)}$, as defined in Section \ref{intro}. Fix a closed point $x\in\Sing(\G)$ and assume that $\tau_{\G,x}\geq r$. 
In case $r=1$ a notion of \emph{presentations} was introduced in Definition \ref{pr_locp}, in terms of suitable morphisms $V^{(d)}\longrightarrow V^{(d-1)}$.  Presentations were, in turn, the tool that enabled us to define the H-functions in the $d-1$-dimensional case, namely $\Hord^{(d-1)}(\G)(x)$. In this section, we address the general case $\tau_{\G,x}\geq r$. In \ref{tau2} we initiate the discussion of presentations which will be ultimately defined in terms of smooth morphisms $V^{(d)}\overset{\beta}{\longrightarrow}V^{(d-r)}$. In Theorem \ref{sepvar} it is proved that such presentations can be chosen in a simplified form, called \emph{simplified presentation}.

These lower dimensional H-functions are introduced in Definition \ref{corolvord}. They appear as the most  natural extension of Definition \ref{defvord} to the case $r>1$. The value $\Hord^{(d-r)}(\G)(x)$ is defined, firstly, in terms of simplified presentations, and finally Theorem \ref{theo} (1) proves that this value is an invariant, and hence it is independent of any choice.

These functions will lead to applications in singularity theory, discussed here in \ref{par714}.
\end{parrafo}

\begin{parrafo}\label{tau2}{\bf The case $\tau\geq 2$}.

Let $\G$ be a differential Rees algebra. Fix a closed point $x\in\Sing(\G)$.  
Suppose that $\tau_{\G,x}\geq 2$ and fix a transversal projection $V^{(d)}\overset{\beta}{\longrightarrow} V^{(d-2)}$. We will proceed essentially in two steps. We shall first indicate how to construct a factorization of the form
\begin{equation}\label{triang}
\xymatrix@C=1pc@R=1.5pc{
&\ \  V^{(d)}\ar[ld]_{\beta_1}\ar[dd]^{\beta}\\
V^{(d-1)}\ar[rd]_{\beta_2}\\
&\ \  V^{(d-2)}
}\end{equation}

This diagram will allow us to define a coarse presentation in the setting of (\ref{pol}). A second step will consist on a suitable change of the previous factorization of $\beta$, that will finally lead us to the construction of  a simplified presentation in the sense of (\ref{pol2}).

In this first step, we make use of \cite[Theorem 6.4]{Ben1}, which ensures that $\tau_{\R_{\G,\beta_1},\beta_1(x)}\geq 1$ as $\tau_{\G,x}\geq 2$. This will enable us to apply twice Proposition \ref{pr_local}. The existence of $\beta_1$ and $\beta_2$ together with \ref{pr_local} will provide us with a coarse presentation, so that $\G$ will have the same integral closure as
$$\calo_{V^{(d)}}[h_\ell(z_1)W^\ell,\Delta^{j_1}_{z_1}(h_\ell(z_1))W^{\ell-j_1},g_m(z_2)W^m,\Delta^{j_2}_{z_2}(g_m(z_2))W^{m-j_2}]_{
1\leq j_1\leq \ell-1,\
1\leq j_2\leq m-1}\odot \beta^*(\R_{\G,\beta})$$
where 
\begin{itemize}
\item $h_\ell(z_1)\in\calo_{V^{(d-1)}}[z_1]$,
\item and $g_m(z_2)\in\calo_{V^{(d-2)}}[z_2]$,
\end{itemize}
 are monic polynomials on the transversal sections $z_1$ and $z_2$, respectively. Let us draw attention to the fact that the coefficients of $h_{\ell}(z_1)$ are in $\calo_{V^{(d-1)}}$, whereas we want to define a notion of slope involving polynomials with coefficients in $\calo_{V^{(d-2)}}$.

This is the second step, addressed in next proposition. It is proved that, locally in \'etale topology,  there is a \emph{simplified presentation}; so that $\G$ and
$$\calo_{V^{(d)}}[f'_n(z'_1)W^n,\Delta_{z'_1}^{j_1}(f'_n)W^{n-j_1},g'_m(z'_2)W^m,\Delta_{z'_2}^{j_2}(g'_m)W^{m-j_2}]_{1\leq j_1\leq n-1,\
1\leq j_2\leq m-1}\odot \beta^*(\R_{\G,\beta})$$
have the same integral closure, where
\begin{itemize}
\item $f'_n(z'_1)\in\calo_{V^{(d-2)}}[z'_1]$,
\item and $g'_m(z'_2)\in\calo_{V^{(d-2)}}[z'_2]$,
\end{itemize}
are monic polynomials on sections $z'_1$ and $z'_2$, respectively. Note that  now both are polynomials with coefficients in $\calo_{V^{(d-2)}}$. 
Moreover, this will be done by changing the factorization of  $\beta$ in (\ref{triang}).

\begin{proposition}\label{separa} 
Let $\G$ be a differential Rees algebra
and let  $x\in\Sing(\G)$ be a closed point at which $\tau_{\G,x}\geq 2$. Then at a suitable neighborhood of $x$, a transversal morphism, say $V^{(d)}\overset{\beta}{\longrightarrow} V^{(d-2)}$, can be constructed in such a way that:
\begin{enumerate}
\item[-] there are global sections  $z_1,z_2$, and $\{dz_1,dz_2\}$ is a basis of the module of $\beta$-differentials, say $\Omega_\beta^1$, and
\item[-] there are two elements $f_nW^n,g_mW^m\in\G$ where:
\begin{itemize}
\item[**] $f_n(z_1)=z_1^n+a_1z_1^{n-1}+\dots+a_{n}\in\calo_{V^{(d-2)},\beta(x)}[z_1]$,
\item[*] $g_m(z_2)=z_2^m+b_1z_2^{m-1}+\dots+b_m\in\calo_{V^{(d-2)},\beta(x)}[z_2] $,
\end{itemize}
\end{enumerate}
and $\G$ has the same integral closure as
\begin{equation}\label{eqlocalpresen2}
\calo_{V^{(d)}}[f_n(z_1)W^n,\Delta_{z_1}^{\alpha}(f_n)W^{n-\alpha},g_m(z_2)W^m, \Delta_{z_2}^{\gamma}(g_m)W^{m-\gamma}]_{1\leq \alpha\leq n-1,\ 1\leq \gamma\leq m-1}\odot\beta^*(\R_{\G,\beta}),
\end{equation}
where $\R_{\G,\beta}$ denotes the elimination algebra, and $\Delta^j_{z_i}$ are as in Remark \ref{rmk43}.
In addition, $\R_{\G,\beta}$ is non-zero whenever $\Sing(\G)$ is not  of co-dimension two locally at $x$.
\end{proposition}

\begin{definition}\label{pr_locp2}({Simplified Presentations for $\tau\geq 2$}).
Let the setting be as above.
We say that 
\begin{equation}\label{eqld2}
s\P(\beta,  z_1,z_2,  f_{n}(z_1), g_{m}(z_2))
\end{equation}
defines a \emph{simplified presentation} of $\G$.
\end{definition}

\noindent{\it Idea of the proof of Proposition \ref{separa}}:
Fix $x\in\Sing(\G)$, so that $\tau_{\G,x}\geq2$. We will first indicate how to produce a diagram as (\ref{triang}). Once this task is achieved, we will construct a scheme $V^{(d-1)}_2$, and two smooth morphisms $V^{(d)}\overset{\delta_1}{\longrightarrow}V_2^{(d-1)}$ and $V_2^{(d-1)}\overset{\delta_2}{\longrightarrow}V^{(d-2)}$, so that the following diagram commutes:
\begin{equation}\label{rombo}
\xymatrix@R=1.7pc@C=1pc{
&\ \ V^{(d)}\ar[dl]_{\beta_1}\ar[dr]^{\delta_1}\ar[dd]^{\beta} &\\
V_1^{(d-1)}\ar[dr]_{\beta_2} &  & V_2^{(d-1)}\ar[dl]^{\delta_2}\\
&\ \ V^{(d-2)}
}
\end{equation}

\noindent{\bf STEP 1}. We shall first construct a morphism $\beta: V^{(d)}\longrightarrow V^{(d-2)}$ together with a diagram (\ref{triang}), and with the following conditions:

\begin{enumerate}
\item[-] There is an element $h_\ell W^{\ell}\in\G$ so that $h_\ell$ is a monic polynomial on $z$ of degree $\ell$, say
$$h_\ell(z)=z^{\ell}+c_1z^{\ell-1}+\dots+c_\ell\in\calo_{V_1^{(d-1)},\beta_1(x)}[z],$$
where $z$ is a global section, so that $\{dz\}$ is a basis of the module of $\beta_1$-differentials.

Here $\R_{\G,\beta_1}(\subset \calo_{V_1^{(d-1)}}[W])$ denotes the elimination algebra corresponding to $\beta_1$. This is a simple algebra at $\beta_1(x)\in\Sing(\R_{\G,\beta_1})$. In fact, $\tau_{\R_{\G,\beta_1},\beta_1(x)}\geq 1$, since $\tau_{\G,x}\geq 2$.

\item[*] There is an element $g_m W^{m}\in\R_{\G,\beta_1}$ so that $g_m$ is a monic polynomial of degree $m$, say
$$g_m(z_2)=z_2^{m}+b_1z_2^{m-1}+\dots+b_m\in\calo_{V^{(d-2)},\beta(x)}[z_2]$$
where $z_2$ is a global section and $\{dz_2\}$ is a basis of the module of $\beta_2$-differentials. Hence $\{dz,dz_2\}$ is a basis of the module of $\beta$-differentials. 

\end{enumerate}

\noindent{\bf STEP 2}. Here we will address the construction of a smooth scheme $V_2^{(d-1)}$, together with morphisms $V^{(d)}\overset{\delta_1}{\longrightarrow}V_2^{(d-1)}$ and $V_2^{(d-1)}\overset{\delta_2}{\longrightarrow}V^{(d-2)}$, so as to complete diagram (\ref{rombo}). The functions $f_n(z_1)$ and $g_m(z_2)$, with the conditions specified in Proposition \ref{separa}, will arise from the construction of the right-hand side in (\ref{rombo}).

Step 1 provides us with an element $g_m(z_2)W^m\in\R_{\G,\beta_1}$. Via the natural inclusion $\R_{\G,\beta_1}\subset \G$ (see Proposition \ref{elim} (1)), the element $g_m(z_2)W^m\in\G$. The smooth scheme $V_2^{(d-1)}$ and the smooth morphism $\delta_1: V^{(d)}\longrightarrow V^{(d-1)}_2$ will be constructed with the conditions that
$g_m(z_2)\in\calo_{V_2^{(d-1)}}[z_2]$, and that $z_2$ defines a section for $\delta_1$. Moreover $\{dz_2\}$ is a basis for the module of $\delta_1$-differentials. 

As $\tau_{\G,x}\geq 2$, we know that $\tau_{\R_{\G,\delta_1},\delta_1(x)}\geq 1$ (\cite[Theorem 6.4]{Ben1}). This provides:

\begin{enumerate}
\item[**] an element $f_nW^n\in\R_{\G,\delta_1}$ which is a monic polynomial, i.e.,
$$f_n(z_1)=z_1^{n}+a_1z_1^{n_1-1}+\dots+a_n\in\calo_{V^{(d-2)}}[z_1].$$
where again $z_1$ is a global section so that $\{dz_1\}$ is a basis of $\Omega_{\delta_2}^1$; and hence $\{dz_1,dz_2\}$  defines a basis of  $\Omega_{\beta}^1$.

\end{enumerate}

Finally check that $f_n(z_1)$ and $g_m(z_2)$ fulfill the condition of the Proposition.

\begin{proof} Here we construct (\ref{rombo}) with the previously required conditions.
\'Etale topology will be used throughout this proof. Fix a smooth scheme $V$ and suppose given a scheme $W$ and a smooth morphism
$$
\xymatrix@R=2.5pc{
V\ar[d]^{\gamma}\\
W
}$$
This setting is preserved in \'etale topology when an \'etale map $\xymatrix{W'\ar[r]^{e} & W}$ is considered. In fact, a commutative diagram arises by taking fiber products:
$$
\xymatrix@R=2.5pc@C=3pc{
V'\ar[r]^{e'}\ar[d]_{\gamma'} & V\ar[d]^{\gamma}\\
W'\ar[r]^{e} & W
}
$$
where $\xymatrix{V'\ar[r]^{e'} & V}$ is an \'etale map, and $\xymatrix{V'\ar[r]^{\gamma'} & W'}$ is smooth. This says that the construction of a scheme $W$ and a smooth morphism $\gamma$ is preserved in \'etale topology, but \emph{only} when lifting \'etale maps in the previous sense (from down-up).
This will be the key point for the construction of the schemes and morphisms previously mentioned.
Recall our general strategy:
\begin{enumerate}
\item[{\bf STEP 1.}] First construct the left hand side of (\ref{rombo}), namely the smooth schemes $V_1^{(d-1)}$, $V^{(d-2)}$ and the morphisms $\beta_1$ and $\beta_2$ with the required conditions. 

\item[{\bf STEP 2.}] Once  the previous data is fixed, complete the diagram (\ref{rombo}) (the right hand side) in such a way that the polynomials $f_n(z_1)$ and $g_m(z_2)$ can be chosen as in Proposition \ref{separa}.
\end{enumerate}

\vspace{0.3cm}

\noindent {\bf 1)} By assumption $\tau_{\G,x}\geq 2(\geq 1)$, so one can find a regular system of parameters $\{x_1,\dots,x_{d}\}$ at $\calo_{V^{(d)},x}$, and an element $h_\ell W^{\ell}\in\G$ so that $h_\ell |_{x_1=\dots=x_{d-1}=0}=u\cdot x_{d}^{\ell}$ for some unit $u\in\calo_{V^{(d)}}/\id{x_1,\dots,x_{d-1}}$.

Consider the smooth morphism 
$$
\xymatrix@C=0pc@R=2.5pc{
V^{(d)}\ar[d]\\
\mathbb{A}^{(d-1)}_k &\!\!\! =\Spec(k[X_1,\dots,X_{d-1}])
}
$$
defined by $X_i\mapsto x_i$ for $i=1,\dots,d-1$. Let $(B,N)$ be the henselization of the local ring $k[X_1,\dots,X_{d-1}]_{\id{X_1,\dots,X_{d-1}}}$. This defines $\Spec(B)\longrightarrow \mathbb{A}^{(d-1)}_k$. Up to multiplication by a unit, the element $h_\ell$ is a monic polynomial of degree $\ell$, i.e.,
$$h_\ell(z)=z^{\ell}+c_1z^{\ell-1}+\dots+c_\ell\in B[z].$$
Now replace $\mathbb{A}_k^{(d-1)}$ by a suitable \'etale neighborhood $V^{(d-1)}_1$ where all the coefficients $c_i$ are global sections. Define $\beta_1$ by taking the fiber product. We abuse the notation and set $\beta_1: V^{(d)}\longrightarrow V_1^{(d-1)}$.

Let $\R_{\G,\beta_1}$ denote the elimination algebra with respect to $\beta_1$. Since $\tau_{\G,x}\geq 2$, then again $\tau_{\R_{\G,\beta_1},\beta_1(x)}\geq 1$ and we repeat the previous argument to define a scheme, say $V^{(d-2)}$, together with a smooth morphism $\beta_2$, so that, a given element $g_mW^m\in \R_{\G,\beta_1}$ can be expressed as
$$g_m(z_2)=z_2^m+b_1z_2^{m-1}+\dots+b_m\in\calo_{V^{(d-2)}}[z_2].$$
This construction might force us to replace the previous scheme $V_1^{(d-1)}$ by an \'etale neighborhood, in particular, $V^{(d)}$ is replaced by an \'etale neighborhood. Set $\beta=\beta_2\circ\beta_1$. 

\vspace{0.2cm}

\noindent {\bf 2)} Fix a regular system of parameters $\{x_1,\dots,x_{d-2}\}$ at $\calo_{V^{(d-2)},\beta(x)}$. It extends to $\{x_1,\dots,x_{d-2},z,z_2\}$ which is a regular system of parameters at $\calo_{V^{(d)},x}$. Set $\mathbb{A}^1=\Spec(k[T])$. We now define a smooth morphism:
$$\xymatrix@R=2.5pc{
V^{(d)}\ar[d]^{\delta'_1}\\
V^{(d-2)}\times \mathbb{A}^1
}$$
with the condition that $pr_1\circ\delta'_1$ yields $\beta=\beta_2\circ\beta_1$; here $pr_1$ is the projection in the first coordinate. If  in addition $T\mapsto z$, then $\delta'_1$ is uniquely defined. Note that $\calo_{V^{(d-2)}}\subset \calo_{V^{(d-2)}\times\mathbb{A}^1}$ and hence $g_m(z_2)\in\calo_{V^{(d-2)}\times \mathbb{A}^1}[z_2]$. Moreover, check that $z_2$ is a global section so that $\{dz_2\}$ defines a basis of $\Omega_{\delta'_1}^{1}$.

Let $\mathcal{H}:=\R_{\G,\delta'_1}\subset\calo_{V^{(d-2)}\times \mathbb{A}^1}[W]$ denote the elimination algebra with respect to $\delta'_1$. Again, since $\tau_{\G,x}\geq 2$, then $\tau_{\mathcal{H},\delta'_1(x)}\geq 1$. The same arguments used before ensures that there are:
\begin{itemize}
\item[-]  an \'etale neighborhood of $V^{(d-2)}$, say $V'^{(d-2)}$, 
\item[-] a smooth morphism $V_2^{(d-1)}\overset{\delta_2}{\longrightarrow}V'^{(d-2)}$ (here $V_2^{(d-1)}$ is an \'etale neighborhood of $V^{(d-2)}\times \mathbb{A}^1$), and 
\item[-] an element $f_nW^n\in\mathcal{H}\ (\subset \G)$ which can be expressed as
$$**\ \ \ \ f_n(z_1)=z_1^{n}+a_1z_1^{n_1-1}+\dots+a_n\in\calo_{V'^{(d-2)}}[z_1],$$
with the required properties.
\end{itemize}

This settles the construction of diagram (\ref{rombo}), and the two polynomials $f_n(z_1)$ and $g_m(z_2)$ fulfill the conditions of Proposition \ref{separa}.
\end{proof}

\begin{definition}\label{sl2}
Fix a Rees algebra $\G$ so that $\tau_{\G,x}\geq 2$ at any closed point $x\in\Sing(\G)$, and a simplified presentation, say $s\P=s\P(\beta,  z_1,z_2,  f_{n_1}(z_1),f_{n_2}(z_2))$, as in (\ref{eqld2}). The \emph{slope of $\G$ relative to $s\P$ at a point $y\in V^{(d-2)}$} is defined as
$$Sl(s\P)(y):=\!\!\!\!\min_{
\xymatrix@R=0pc@C=0pc{
\scriptstyle 1\leq j_1\leq n_1\\
\scriptstyle 1\leq j_2\leq n_2}
}\!\!\!\! \Big\{\frac{\nu_y(a_{j_1})}{j_1},\frac{\nu_y(b_{j_2})}{j_2},\ord(\R_{\G,\beta})(y)\Big\}.$$
Here $\nu_y$ denotes the order at the local regular ring $\calo_{V^{(d-2)},y}$ and $\ord\R_{\G,\beta}$ is the order function the Rees algebra $\R_{\G,\beta}$ as defined in (\ref{ordG}).
\end{definition}

\begin{definition}\label{nf2} Let $\G$ be a Rees algebra, so that $\tau_{\G,x}\geq 2$ for all $x\in\Sing(\G)$. A simplified presentation $s\P=s\P(\beta,  z_1,z_2,  f_{n_1}(z_1),f_{n_2}(z_2))$ is said to be in \emph{normal form at $y\in V^{(d-2)}$} if one of the following conditions hold:

\begin{itemize}
\item Either $Sl(s\P)(y)=\ord(\R_{\G,\beta})(y)$,
\item or $Sl(s\P)(y)= Sl(f_{n_1}(z_1)W^{n_1})(y)<\ord(\R_{\G,\beta})(y)$ and $\win_{y}(f_{n_1}(z_1))$ is not an $n_1$-th power,
\item or  $Sl(s\P)(y)= Sl(f_{n_2}(z_2)W^{n_2})(y)<\ord(\R_{\G,\beta})(y)$ and $\win_{y}(f_{n_2}(z_2))$ is not an $n_2$-th power.
\end{itemize}
\end{definition}

It will be shown in Theorem \ref{theo} that the rational number $Sl(s\P)(y)$ is entirely determined by the weak equivalence class of $\G$, whenever $s\P$ is in normal form at $y\in\beta(\Sing(\G))$. This will lead to the definition of an H-function along points of $\Sing(\G)$.
\end{parrafo}

\begin{parrafo}{\bf The case $\tau\geq e$}.\label{taur}

We address here the case $\tau\geq e$, now for arbitrary $e$. It parallels the previous results for $e=2$.

\begin{theorem}\label{sepvar}
Let $\G$ be a Rees algebra so that $\tau_{\G,x}\geq e$ at a closed point $x\in\Sing(\G)$. Then, at  a suitable \'etale neighborhood of $x$, a transversal morphism, say $V^{(d)}\overset{\beta}{\longrightarrow}V^{(d-e)}$, can be defined together with:
\begin{itemize}
\item[-] global functions $z_1,\dots,z_e$, and $\{dz_1,\dots,dz_e\}$ is a basis of $\Omega^1_\beta$ (module of $\beta$-relative differentials),  
\item[-] integers $n_1,\dots,n_e\in\mathbb{Z}_{>0}$, and
\item[-] elements $f_{n_1}W^{n_1},\dots,f_{n_e}W^{n_e}\in\G$, where
\end{itemize}
\begin{equation}\label{polsep}
\begin{array}{l}
f_{n_1}(z_1)= z_1^{n_1}+a_1^{(1)}z_1^{n-1}+\dots+a_{n_1}^{(1)}\in \calo_{V^{(d-e)}}[z_1],\\
 \ \ \ \vdots\\
f_{n_e}(z_e)=z_r^{n_e}+a_1^{(e)}z_r^{n_e-1}+\dots+a_{n_e}^{(e)}\in \calo_{V^{(d-e)}}[z_e],
\end{array}
\end{equation}
(for some $a_i^{(j)}$ global functions on $V^{(d-e)}$). Moreover, the previous data can be defined with the condition that $\G$ has the same integral closure as 
\begin{equation} \label{eqdpl2} 
\calo_{V^{(d)}}[f_{n_i}(z_i)W^{n_i},\Delta_{z_i}^{j_i}(f_{n_i}(z_i))W^{n_i-j_i}]_{1\leq j_i\leq n_i-1, \ i=1,\dots,e}\odot\beta^*(\R_{\G,\beta}).
\end{equation}
Here, $\R_{\G,\beta}\subset\calo_{V^{(d-e)}}[W]$ is the elimination algebra defined in terms of $\beta$ as in Proposition \ref{elim}, and $\Delta^j_{z_i}$ are as in Remark \ref{rmk43}.
\end{theorem}

\begin{proof}
The proof follows straightforward from that of Proposition \ref{separa}.
\end{proof}

\begin{definition}\label{pr_locpr}{\rm (Simplified presentations)}.
Let the setting be as in Theorem \ref{sepvar}. The data 
\begin{equation}\label{eqldr}
s\P \big(\beta,  z_1,\dots z_e,  f_{n_1}(z_1),\dots,f_{n_e}(z_e)\big)
\end{equation}
which fulfills the conditions of Theorem \ref{sepvar} is said to be a \emph{simplified presentation}.
\end{definition}

The following Definitions extend those in \ref{sl2} and \ref{nf2}.
\begin{definition}\label{slr}
Let $\G$ be a differential Rees algebra $\G$ so that $\tau_{\G,x}\geq e$ at a closed point $x\in\Sing(\G)$. Fix, at a neighborhood of $x$,  a simplified presentation, say $s\P=s\P(\beta,  z_1,\dots, z_e,  f_{n_1}(z_1),\dots, f_{n_e}(z_e))$, as in (\ref{eqldr}). The \emph{slope of $\G$ relative to $s\P$ at a point $y\in V^{(d-e)}$} is the rational number defined as
\begin{equation}\label{eqslgen}
Sl(s\P)(y):=\!\!\!\!\min_{
\xymatrix@R=0pc@C=0pc{
\scriptstyle 1\leq j_i\leq n_i\\
\scriptstyle 1\leq i\leq e}
}\!\!\!\! \Big\{\frac{\nu_y(a^{(i)}_{j_i})}{j_i},\ord(\R_{\G,\beta})(y)\Big\}.
\end{equation}
\end{definition}

\begin{definition}\label{rnormal} Let $\G$ be a differential Rees algebra in the conditions of the previous definition. A simplified presentation $s\P=s\P(\beta,  z_i,  f_{n_i}(z_i))_{1\leq i\leq e}$ is said to be in \emph{normal form at $y\in V^{(d-e)}$} if one of the following conditions holds:

\begin{itemize}
\item Either $Sl(s\P)(y)=\ord(\R_{\G,\beta})(y)$,
\item or for some index $1\leq i\leq e$, $Sl(s\P)(y)= Sl(f_{n_i}(z_i)W^{n_i})(y)<\ord(\R_{\G,\beta})(y)$ and $\win_{y}(f_{n_i}(z_i))$ is not an $n_i$-th power. 
\end{itemize}
\end{definition}

The next theorem will show that given a simplified presentation $s\P$ in normal form at $y\in\beta(\Sing(\G))$, the rational value $Sl(s\P)(y)$ is an invariant.

\begin{theorem}\label{theo}
Let $\G$ be a $\beta$-differential Rees algebra (e.g. a differential Rees algebra) with the property that $\tau_{\G,x}\geq e$ for all $x\in\Sing(\G)$. Fix a point $x\in\Sing(\G)$ and assume that there is a simplified presentation $s\P$ which is in normal form at $\beta(x)$ (see Definition \ref{rnormal}). 
\begin{enumerate}
\item The rational number $q=Sl(s\P)(\beta(x))$ in (\ref{eqslgen}) is  entirely determined by the weak equivalence class of $\G$.
\item $\G$ and $\calo_{V^{(d)}}[f_{n_1}W^{n_1},\dots,f_{n_r}W^{n_e}]\odot \beta^{*}(\R_{\G,\beta})$ are weakly equivalent. 
\end{enumerate}
\end{theorem}

\begin{proof}
\ \ \ (1) The proof is similar to that of Theorem \ref{thmqG}, considering now, in Stage B,  blow-ups at centers of codimension $e+1$ instead of codimension $2$.

(2) Theorem \ref{sepvar} ensures that $s\P$ can be chosen so that $\G$ has the same integral closure as
$$
\G_1=\calo_{V^{(d)}}[f_{n_i}(z_i)W^{n_i},\Delta^{j_i}_{z_i}(f_{n_i}(z_i))W^{n_i-j_i}]_{1\leq j_i\leq n_i-1, \ i=1,\dots,e}\odot\beta^*(\R_{\G,\beta}).
$$
Theorem \ref{rmkweq} (1) asserts that $\G$ and $\G_1$ are weakly equivalent. Theorem \ref{rmkweq} (2) ensures that $\G_1$ and $\calo_{V^{(d)}}[f_{n_1}W^{n_1},\dots,f_{n_e}W^{n_e}]\odot \beta^{*}(\R_{\G,\beta})$ are weakly equivalent.
\end{proof}

\begin{definition}\label{corolvord}
The $d-e$-dimensional H-function is defined, in the setting of the previous theorem, say
$$\Hord^{(d-e)}(\G):\Sing(\G)\longrightarrow \mathbb{Q}_{\geq 0}$$ 
by setting:
$$\Hord^{(d-e)}(\G)(x)=Sl(s\P)(\beta(x)),$$
where, $s\P$ is a simplified presentation in normal form at $\beta(x)$.
\end{definition}

Theorem \ref{theo} ensures that this value is well-defined. Namely, that it is independent of the choice of the smooth morphism $\beta$ and of the simplified presentation.
\end{parrafo}

\section{Main properties of H-functions and proof of Theorem \ref{IntroTH}}

\begin{parrafo}{\bf A particular kind of simplified presentations: $p$-presentations}\label{par714}. 

Elimination theory is largely sustained on the presence of monic polynomials, indeed elimination algebras are defined in terms of this monic polynomials, the first example appears in Theorem \ref{sepvar}. Astonishing properties arises when these monic polynomials have as degrees powers of the characteristic. Namely, when $n_i=p^{\ell_i}$ in (\ref{polsep}). In fact, in this case, there is a surprising interplay between the coefficients and the elimination algebra.

Fix a $d$-dimensional scheme $V^{(d)}$  smooth over a perfect field $k$ together with a differential Rees algebra, say $\G$,  over $V^{(d)}$. Fix a transversal projection $V^{(d)}\overset{\beta}{\longrightarrow} V^{(d-1)}$. As $\G$ is a differential algebra, it is also a $\beta$-differential algebra.   In this case, locally at any point $x\in\Sing(\G)$, the $\beta$-differential structure of $\G$ enables us to consider a simplified presentation $s\P$  with  integers of the form:
$n_1=p^{\ell_1}< n_2=p^{\ell_2}< \dots< n_e=p^{\ell_e}$, where $p$ denotes the characteristic of $k$. This particular simplified presentations will be called \emph{$p$-presentations} and denoted by
\begin{equation}\label{psp}
p\P=p\P(\beta,z_i,f_{p^{\ell_i}}(z_i)=z_i^{p^{\ell_i}}+a_1^{(i)}z^{p^{\ell_i}-1}+\dots+a_{p^{\ell_i}})_{1\leq i\leq e}.
\end{equation}
The exponents $\ell_1\leq \ell_2\leq  \dots\leq \ell_e$ are closely related with invariants studied by Hironaka in \cite{Hironaka70}, and also related with other invariants introduced by Kawanoue and Matsuki in \cite{Kaw} and \cite{KM}.

The notion of $p$-presentations were introduced in \cite{BeV1} for the case $e=1$. They were denoted by
$$p\P=p\P(\beta:V^{(d)}\longrightarrow V^{(d-1)},z,f_{p^\ell}(z)=z^{p^\ell}+a_1z^{p^\ell-1}+\dots+a_{p^\ell}).$$
There it is shown that:
\begin{enumerate}
\item given $x\in\Sing(\G)$, the $p$-presentation can be modified into a new $p$-presentation which is in normal form at $\beta(x)$ (see Definition \ref{nfG}),
\item and 
\begin{equation}\label{Hordp1}
\Hord^{(d-1)}(\G)(x)=\min\Big\{\frac{\nu_{\beta(x)}(a_{p^\ell})}{p^\ell},\ord(\R_{\G,\beta})(\beta(x))\Big\},
\end{equation}
where the right hand side is defined in terms of a $p$-presentation $p\P=p\P(\beta,z,f_{p^\ell}(z))$ in normal form at $\beta(x)$ (see \cite[Theorem 4.6. and Corollary A.9]{BeV1}). 
\end{enumerate}

This simplifies the expression in (\ref{eqdefvord}), as only the coefficient $a_{p^\ell}$ appears in this formulation. Surprisingly, the contribution of the intermediate coefficients of $f_{p^\ell}(z)=z^{p^\ell}+a_1z^{p^\ell-1}+\dots+a_{p^\ell}$ is somehow encoded in the elimination algebra $\R_{\G,\beta}$. 

The previous two properties extend easily to the general case  $\tau\geq e$, via Theorem \ref{sepvar} together with presentations as in (\ref{psp}). Namely:
\begin{enumerate}
\item Given $x\in\Sing(\G)$,  the $p$-presentation $p\P$ in  (\ref{psp}) can be modified into another $p$-presentation, which is in normal form at $\beta(x)$. Let us emphasize here that $p$-presentations are particular forms of simplified presentations.

\item Assume now that $p\P$ is in normal form at $\beta(x)$, then
\begin{equation}\label{pHord}
\Hord^{(d-e)}(\G)(x)=\min_{1\leq j\leq e}\Big\{\frac{\nu_{\beta(x)}\big(a_{p^{\ell_j}}^{(j)}\big)}{p^{\ell_j}},\ord(\R_{\G,\beta})(\beta(x))\Big\}.
\end{equation}
\end{enumerate}

\end{parrafo}

\begin{parrafo}{\bf $p$-presentations and monoidal transformations}.

Insofar $\G$ was assumed to be a differential Rees algebra. It is in this context in which Theorem \ref{sepvar} applies, namely, the theorem ensures that simplified presentations exists for these Rees algebras.
Our task now is to discuss the notion of presentations for a fixed sequence of transformations, say (\ref{conmut}). Namely, to attach a simplified presentation to $\G_r$, locally at $x\in\Sing(\G_r)$, in terms of  $\beta_r$. This will be done by successive ``transformations'' of simplified presentations.

Fix a $p$-presentation of $\G$ as in (\ref{psp}). 
In \cite{BeV1} a notion of transformations of a $p$-presentation was studied for the  case $e=1$: Given a $p$-presentation of $\G$ over $V^{(d)}$, say $p\P=p\P(\beta,z,f_{p^\ell})$, a new $p$-presentation, say $p\P_1=p\P_1(\beta_1,z_1,f^{(1)}_{p^\ell})$, is defined in terms of $p\P$ and of the monoidal transformation. This is denoted by
$$\xymatrix@R=-0.2pc@C=3pc{
p\P & p\P_1\\
V^{(d)} & V^{(d)}_1\ar[l]_{\pi_{C}}.
}$$
This result  readily extends  to the case $\tau\geq e$, again thanks to Theorem \ref{sepvar} applied here to simplified presentations in the setting of (\ref{psp}).

Theorem \ref{theo} has now a natural formulation for the $\beta_r$-differential Rees algebra $\G_r$ in (\ref{conmut}). Also Definition \ref{corolvord} can be stated for $\G_r$. Consequently the $d-e$-dimensional H-function of $\G_r$, say
$$\Hord^{(d-e)}(\G_r):\Sing(\G_r)\longrightarrow \mathbb{Q}_{\geq 0},$$
is defined, and again Theorem \ref{theo} ensures that the function is intrinsic.

\end{parrafo}

\begin{parrafo}{\bf On tamed H-functions and the proof of Theorem \ref{IntroTH}}.

Finally we address the proof of Theorem \ref{IntroTH}. Firstly we will show that the numerical conditions in (\ref{eqTHs}) can be easily simplified. Once this point is settled, Theorem \ref{IntroTH} will be reformulated as Theorem \ref{thmsmc}.

Firstly recall the construction of the monomial algebra introduced in (\ref{leqmn}), where a monomial algebra, say
\begin{equation}\label{eqmon}
\m_r W^s=\calo_{V_r^{(d-e)}}[I(H_1)^{h_1}\dots I(H_r)^{h_r}W^s],
\end{equation}
is assigned to any sequence (\ref{conmut}), by setting, for $i=0,\dots,r-1$:
\begin{equation}\label{eqhmon}
\frac{h_{i+1}}{s}=\Hord^{(d-e)}(\G_i)(\xi_{Y_i})-1,
\end{equation}
and where $\xi_{Y_i}$ denotes the generic point of $Y_i$,  the center of the monoidal transformation.

\begin{definition}\label{defsmc}
Let $\G$ be  a Rees algebra so that $\tau_\G\geq e$. Let the setting be as in \ref{cuad}, and assume that (\ref{conmut}) is defined so that the elimination algebra of $\G_r$ is monomial in the sense of (\ref{mon}). The Rees algebra $\G_r$ is said to be in the \emph{strong monomial case} if at any closed point $x\in\Sing(\G_r)$:
\begin{equation}\label{eqsmcD}
\Hord^{(d-e)}(\G_r)(x)=\ord(\m_rW^s)(x).
\end{equation}
\end{definition}

\begin{remark}\label{rmksmc}
Let us remark that both in Theorem \ref{IntroTH}, and in Definition \ref{IntroDef}, the strong monomial case was defined in terms of two possible numerical conditions, namely 
$$
 \Hord^{(d-e)}(\G_r)(x)=\ord(\m_rW^s)(x)\ \ \hbox{ or }\ \ \Hord^{(d-e)}(\G_r)(x)= \Ord^{(d-e)}(\G_r)(x)=\ord((\R_{\G,\beta})_r)(\beta_r(x)),
$$
whereas (\ref{eqsmcD}) neglects the second condition. Some clarification is in order to justify this fact. We have assumed, as we may, that the elimination algebra of $\G_r$ is monomial. Namely, that
\begin{equation}\label{eqma}
(\R_{\G,\beta})_r=\calo_{V_r^{(d-e)}}[I(H_1)^{\alpha_1}\dots I(H_r)^{\alpha_r}W^s].
\end{equation}
We observe first that the monomial algebra $\m_rW^s$ in (\ref{eqmon}) divides $(\R_{\G,\beta})_r$ (i.e., $h_i\leq \alpha_i$ for $i=1,\dots,r$). This fact is guaranteed by (\ref{eqhmon}) and (\ref{pHord}), together with the law of transformations for Rees algebras.

Finally, as $\m_rW^s$ divides $(\R_{\G,\beta})_r$, the condition $\Hord^{(d-e)}(\G_r)(x)= \ord(\R_{\G,\beta})_r)(\beta_r(x))$ ensures that $\m_r W^s=(\R_{\G,\beta})_r$ locally at $x$, and in particular that $ \Hord^{(d-e)}(\G_r)(x)=\ord(\m_rW^s)(x)$.

So both numerical conditions in (\ref{eqTHs}) reduces to the unique condition in (\ref{eqsmcD}).
\end{remark}

\begin{theorem}\label{thmsmc}
Let the setting be as in Definition \ref{defsmc}. Assume that $\G_r$ is in the strong monomial case at any closed point $x\in\Sing(\G_r)$. A combinatorial resolution of the monomial algebra $\m_rW^s$ can be naturally lifted to a resolution of $\G_r$.
\end{theorem}

\begin{proof}
Our proof will strongly rely on the existence of $p$-presentations (see \ref{par714}), say $s\P=s\P(\beta,z_i,f_{p^{\ell_i}}(z_i)=z_i^{p^{\ell_i}}+a_1^{(i)}z^{p^{\ell_i}-1}+\dots+a_{p^{\ell_i}})_{1\leq i\leq e}$, which we may assume, in addition, that are in normal form at $\beta(x)$ ($x\in\Sing(\G_r)$).

The monomial algebra $\m_rW^s$ in (\ref{eqmon}) relates to the $p$-presentation in the following manner:
\begin{itemize}
\item $a_{j_i}^{(i)}W^{j_i}\in \m_rW^s\quad \hbox{for }j_i=1,\dots p^{\ell_i},\ i=1,\dots,e$.
Meaning that the coefficients $a_{j_i}^{(i)}W^{j_i}$ are in the integral closure of $\m_rW^s$. We express this fact by saying that $\m_rW^s$ \emph{divides} the coefficients.

\item $\m_rW^s$ divides the elimination algebra $(\R_{\G,\beta})_r$ in (\ref{eqma}) (see Remark \ref{rmksmc}).
\end{itemize}

Let us address now the numerical conditions defined in terms of the H-functions $\Hord^{(d-e)}$ at points of $\Sing(\G_r)$. Firstly recall the simplification given in (\ref{pHord}) in which the $\Hord$ relies on the elimination algebra or on the constant terms of the $e$ polynomials, say $a_{p^{\ell_i}}^{(i)}$.

(A) Assume that $\Hord^{(d-e)}(\G_r)(x)=\ord((\R_{\G,\beta})_r)(\beta_r(x))$ at  $x\in\Sing(\G_r)$. In this case, Remark \ref{rmksmc}) says that $\m_r W^s=(\R_{\G,\beta})_r$, this, in turn, ensures that all coefficients $a_{j_i}^{(i)}W^{j_i}\in(\R_{\G,\beta})_r$. These conditions, together with the fact that $f_{p^{\ell_i}}(z_i)W^{p^{\ell_i}}\in\G_r$,  say that, up to integral closure, $z_iW\in\G_r$. In particular, this guarantees that $z_1=\dots=z_e=0$ has maximal contact with $\G_r$ at $x$, and hence in an open neighborhood of $x$. Hence, at a neighborhood of $x$ the Theorem can be dealt with as in the case of characteristic zero.

(B) Assume that $\Hord^{(d-e)}(\G_r)(x)<\ord((\R_{\G,\beta})_r)(\beta_r(x))$ at  $x\in\Sing(\G_r)$. Then, there is an index $j$ for which $\Hord^{(d-e)}(\G_r)(x)=
\displaystyle\frac{\nu_{\beta_r(x)}(a_{p^{\ell_j}}^{(j)})}{p^{\ell_j}}\ \big(=\ord(\m_rW^s)(x)\big)$. This equality shows that 
$$a_{p^{\ell_j}}^{(j)}W^{p^{\ell_j}}=\m_rW^s,$$
meaning that both algebras have the same integral closure. This ensures that, up to multiplication by a unit, $a_{p^{\ell_j}}^{(j)}$ can be taken as a monomial.

Fix now a smooth permissible center $C$ with generic point $y$, so that $x\in C$. A property of $p$-presentations is that they can be chosen so as to be in normal form simultaneously at $x$ and $y$ (see \cite[Proposition 5.8]{BeV1}). We now compute $\Hord^{(d-e)}(\G_r)(y)$. We claim that 
$$\Hord^{(d-e)}(\G_r)(y)=\frac{\nu_{\beta_r(y)}(a_{p^{\ell_j}}^{(j)})}{p^{\ell_j}}$$
for the same index $j$ we have taken before. In fact, $\ord(\m_rW^s)(y)=\frac{\nu_{\beta_r(y)}(a_{p^{\ell_j}}^{(j)})}{p^{\ell_j}}$ and since all coefficients and the elimination algebra are divisible by $\m_rW^s$, then 
$$\Hord^{(d-e)}(\G)(y)=\min_{1\leq j\leq e}\Big\{\frac{\nu_{\beta_r(y)}\big(a_{p^{\ell_j}}^{(j)}\big)}{p^{\ell_j}},\ord((\R_{\G,\beta})_r)(\beta_r(x))\Big\}
=\frac{\nu_{\beta_r(y)}(a_{p^{\ell_j}}^{(j)})}{p^{\ell_j}}$$
Summarizing, there is a particular index $j$ with the following properties:
\begin{itemize}
\item $j$ provides the value of the H-function at $x$ and at $y$. Namely,
$$\Hord^{(d-e)}(\G)(x)=\frac{\nu_{\beta_r(x)}(a_{p^{\ell_j}}^{(j)})}{p^{\ell_j}}\quad \hbox{ and }\quad\Hord^{(d-e)}(\G)(y)=\frac{\nu_{\beta_r(y)}(a_{p^{\ell_j}}^{(j)})}{p^{\ell_j}}$$
\item the role of this $j$ is stable under transformations.
\end{itemize}

In this way, Theorem \ref{sepvar} reduces the proof of the theorem to the case of only a unique index $j$. A simple proof of Theorem \ref{thmsmc} under this reduced conditions was given in \cite{BeV1}.
\end{proof}

The embedded resolution of surfaces treated in \cite[Part III]{BeV1}  for the hypersurfaces case,  can now be extended to prove embedded desingularization  of arbitrary $2$-dimensional schemes (over perfect fields). This provides a desingularization of $2$-dimensional schemes \emph{a la Hironaka}, namely by applying successive monoidal transformations along the highest Hilbert-Samuel stratum.
\end{parrafo}


\end{document}